\documentclass[11pt]{article}
\usepackage{amsmath,amsthm,amsfonts,amssymb,amscd}
\usepackage[scr]{rsfso}
\usepackage{stackrel}
 \usepackage{amsmath}
\usepackage{enumerate} 
\usepackage{color}
\usepackage{float}
\usepackage{soul}
\usepackage{graphicx}
\usepackage{mathpazo}
\usepackage{fancyhdr}
\usepackage{empheq}
\usepackage[margin=1in]{geometry}

\def\XX{\mathbb{X}}
\def\YY{\mathbb{Y}}

\def\Limsup{\mathop{{\rm Lim}\,{\rm sup}}}

\def\tto{\rightrightarrows}
\def\Hat{\widehat}

\def\Tilde{\widetilde}
\def\Bar{\overline}
\def\ra{\rangle}
\def\la{\langle}
\def\ve{\varepsilon}
\def\B{\mathbb{B}}
\def\h{\hfill\Box}
\def\R{\mathbb{R}}
\def\ox{\bar{x}}
\def\OX{\Bar{X}}
\def\oy{\bar{y}}
\def\OY{\Bar{Y}}

\def\ov{\bar{v}}
\def\OV{\Bar{V}}
\def\ow{\bar{w}}

\def\ou{\bar{u}}
\def\OU{\Bar{U}}

\def\op{\bar{p}}

\def\cone{\mbox{\rm cone}\,}
\def\ri{\mbox{\rm ri}\,}

\def\Im{\mbox{\rm Im}\,}

\def\gph{\mbox{\rm gph}\,}

\def\rank{\mbox{\rm rank}\,}

\def\dom{\mbox{\rm dom}\,}
\def\Ker{\mbox{\rm Ker}\,}

\def\cl*co{\mbox{\rm cl}^*\mbox{\rm co}\,}

\def\cl{\mbox{\rm cl}\,}

\def\h{\hfill\triangle}
\def\dn{\downarrow}
\def\O{\Omega}

\def\st{\stackrel}
\def\oR{\Bar{\R}}

\def\hs7{\hspace*{7pt}}

\def\Id{\mathbb{I}}

\renewcommand{\theequation}{\thesection.\arabic{equation}}

\def\h{\hfill\Box}
\def\kk{\kappa}
\begin{document}

\newtheorem{Theorem}{Theorem}[section]
\newtheorem{Conjecture}[Theorem]{Conjecture}
\newtheorem{Proposition}[Theorem]{Proposition}
\newtheorem{Remark}[Theorem]{Remark}
\newtheorem{Lemma}[Theorem]{Lemma}
\newtheorem{Corollary}[Theorem]{Corollary}
\newtheorem{Definition}[Theorem]{Definition}
\newtheorem{Example}[Theorem]{Example}
\newtheorem{Fact}[Theorem]{Fact}
\newtheorem*{pf}{Proof}
\renewcommand{\theequation}{\thesection.\arabic{equation}}
\normalsize
\normalfont
\medskip
\def\endproof{$\h$\vspace*{0.1in}}

\title{\bf Isolated Calmness in Regularized Convex Optimization}
\date{}
\author{TRAN T. A. NGHIA\footnote{Department of Mathematics and Statistics, Oakland University, Rochester, MI 48309, USA; email: nttran@oakland.edu} \and HUY N. PHAM\footnote{Department of Mathematics and Statistics, Oakland University, Rochester, MI 48309, USA; email:  hnpham@oakland.edu}}

\maketitle

\begin{center}
    Dedicated to Professor Hoàng Xuân Phú on the occasion of his 70th birthday  
\end{center}

\begin{quote}
{\small \noindent {\bf Abstract.} This paper studies the isolated calmness of the optimal solution mapping and the associated Lagrange system for regularized convex composite optimization problems. Several necessary and sufficient conditions for this property are established. These conditions are geometric in nature and relatively simple to verify. To support the analysis, we also develop a so-called {\em zero-product property} for second-order structures, namely the graphical derivative of the subgradient mapping of convex functions.

}

\vspace{0.1in}

\end{quote}
\vspace{-0.2in}

\section{Introduction}
\setcounter{equation}{0}

One of the most popular methods to recover a signal $x_0$ in a Euclidean space $\XX$ from its observation $b$ in another Euclidean space $\mathbb{E}$ over a linear operator $\Phi:\XX\to \mathbb{E}$, i.e., $b\approx\Phi x_0$ is solving the following least-squares regularized problem
\begin{equation}\label{eq:RLS}
\min_{x\in \XX}\quad \frac{1}{2\mu}\|\Phi x-b\|^2+h(x),    
\end{equation}
where $\mu >0$  is known as the regularization (scaling, tuning) parameter and $h:\XX\to \R\cup\{\infty\}$ is a proper lower semi-continuous (l.s.c.) convex function (a.k.a. regularizer), usually determined due to some desirable prior information of signal $x_0$. For example, if we expect to recover a {\em sparse} (or {\em low-rank}) signal, a reasonable choice for $h$ is the $\ell_1$ norm in $\XX=\R^n$ (or the nuclear norm in $\XX=\R^{m\times n}$, respectively). In many situations, a signal may be not sparse, but it is after a linear transformation, e.g., the {\em Fourier transformation} of a signal in compressed sensing or the {\em discrete gradient operator} of an image in imaging. A modified model to match these situations is    
\begin{equation}\label{eq:MLS}
\min_{x\in \XX}\quad \frac{1}{2\mu}\|\Phi x-b\|^2+g(Kx),    
\end{equation}
where $K:\XX\to \YY$ is a linear operator and $g:\YY\to \R\cup\{\infty\}$ is a proper l.s.c. convex regularizer. Both problems \eqref{eq:RLS} and ~\eqref{eq:MLS} have been used widely in optimization, statistics, signal and image processing, as well as machine learning. 

The positive tuning parameter $\mu$  represents for the trade-off between the data fitting error $\|\Phi x-b\|$ and the prior regularization $h(x)=g(Kx)$. As $b$ is usually a noised observation of $\Phi x_0$, it is another parameter in this model. It is natural to question the stability of the optimal solution mapping of problem \eqref{eq:RLS} (or \eqref{eq:MLS}) with these two parameters defined by
\begin{equation}\label{eq:OS}
    S_{LS}(b,\mu):={\rm argmin}\,\left\{\frac{1}{2\mu}\|\Phi x-b\|^2+ h(x)|\; x\in \XX\right\}.
\end{equation}
Different kinds of Lipschitz stability of this solution mapping have been studied recently in \cite{BBH23,BBH24,CHNS25,HYZ24,MWY24,N24}. In particular, \cite{BBH23,CHNS25,N24} established necessary and sufficient conditions for the situation that the solution mapping $S$ is locally single-valued and Lipschitz continuous around a given pair $(\bar b,\bar \mu)\in \mathbb{E}\times \R_{++}$  for  different regularizers $h$ such as the $\ell_1$ norm \cite{BBH23}, the $\ell_1/\ell_2$ norm and the nuclear norm \cite{N24}, and the broad class of $\mathcal{C}^2$-reducible dual regularizers \cite{CHNS25}. The single-valued Lipschitz continuity of this mapping is also behind the study of its sensitivity analysis in \cite{BLPS21,BPS22,VDFPD17}. On the other hand, \cite{HYZ24,MWY24} investigated the (global) Lipschitz continuity of $S$ as a set-valued mapping, when $h(x)$ is the $\ell_1$ norm in $\R^n$. 

Our paper mainly studies the Lipschitz continuity {\em at} a given pair $(\bar b,\bar \mu)\in \mathbb{E}\times \R_{++}$, which is usually referred as the {\em isolated calmness} of $S$ at $(\bar b,\bar \mu)$ for some $\ox\in S(\bar b,\bar \mu)$ in the sense that there exists a neighborhood $U\times V\times  W$ of $(\ox,\bar b,\bar \mu)$ in $\XX\times \mathbb{E}\times \R_{++}$  and a constant $\kk>0$ such that 
\begin{equation}\label{eq:Ca}
    S(b,\mu)\cap U\subset \ox+\kk(\|b-\bar b\|+|\mu-\bar \mu|)\B\quad \mbox{for all}\quad (b,\mu)\in V\times W,
\end{equation}
where $\B$ is the Euclidean unit ball in $\XX$.  This property means that any optimal solution of problem~\eqref{eq:RLS} with small perturbation $(b,\mu)$ of $(\bar b,\bar \mu)$ in a neighborhood of $\ox$ can be approximated by $\ox$ with a ``linear rate'' $(\|b-\bar b\|+|\mu-\bar \mu|)$. It plays a significant role in the stability theory for optimization problems and convergence analysis of algorithms; see, e.g., the monograph \cite{DR09,I17,KK02} for the detailed study of this property  and its applications to optimization and algorithms. The mapping $S$ does not need to be single-valued around $(\bar b,\bar \mu)$, but $\ox$ has to be the isolated point of $S(\bar b,\bar \mu)$. It is well-known that the isolated calmness  can be characterized fully via the so-called {\em graphical derivative} of $S$ at $(\bar b,\bar \mu)$ for $\ox$, a generalized differential structure for set-valued mappings; cf. \cite[Definition~8.33]{RW98}. This characterization, usually referred as Levy-Rockafellar criterion \cite{L96} is the main mechanism for studying the isolated calmness. But it depends on the computation of the graphical derivative, which could be complicated in many problems. In particular, it is not clear to us how to compute the graphical derivative of the solution mapping $S$ in \eqref{eq:MLS} fully. Although this mapping can be written as a {\em generalized equation} as in \cite{L96}
\[
S_{LS}(b,\mu)=\left\{x\in \XX|\; 0\in \frac{1}{\mu}\Phi^*(\Phi x-b)+\partial h(x)\right\},
\]
 only an upper estimate for the graphical derivative of this mapping is known \cite[Corollary~3.5]{L96}. Actually, our paper does not try to compute the graphical derivative of this set-valued mapping. Our approach relies on the {\em linearization} of this generalized equation and the formula of the {\em kernel} of the graphical derivative of the {\em subdifferential mapping} without fully calculating this graphical derivative.  The isolated calmness for different  variational systems and various choices of parameters (tilt and canonical perturbations) were studied in  \cite{BM24,CS18,DSZ17,LPB21,MMS22,ZZ16}, mostly on Karush-Kuhn-Tucker (KKT)/Lagrange systems of constrained and composite (nonlinear) optimization problems, which was initiated by Dontchev and Rockafellar in \cite{DR97} for nonlinear programming. It is worth noting that the parameter pair $(b,\mu)$ studied in our paper varies from the  tilt and canonical perturbations used in those papers. Moreover, despite the fact that our function $h$ will be assumed to satisfy the {\em quadratic growth condition}, it is in a general setting and {\em partially} covers the case of {\em nuclear norm} regularizer in \cite{CS18}, {\em convex piecewise linear-quadratic} functions in \cite[Section 8]{MMS22}, and the indicator to polyhedral sets \cite{DR97} and positive semidefinite cone \cite{ZZ16}.

{\bf Our Contribution.}
For a given pair $(\bar b,\bar \mu)\in \mathbb{E}\times \R_{++}$, we show that the optimal solution mapping $S$ in \eqref{eq:OS} is isolated calm at $(\bar b,\bar \mu)$ for $\ox\in S(\bar b, \bar \mu)$ if and only if 
\begin{equation}\label{eq:Char}
\Ker\Phi\cap T_{\partial h^*(\ov)}(\ox)=\{0\}\quad \mbox{with}\quad \ov:=-\frac{1}{\bar\mu}\Phi^*(\Phi \ox-\bar b),
\end{equation}
where $T_{\partial h^*(\ov)}(\ox)$ represents for the {\em tangent cone} to the set $\partial h^*(\ov)$ at $\ox$, provided that the function $h$ satisfies the {\em quadratic growth condition} \cite{AG08,BS00} at $\ox$ for $\ov$; see our Corollary~\ref{cor:Equ1}. The assumption of quadratic growth condition on the regularizer $h$ is not restrictive; many popular regularizers satisfy this condition including {\em convex piecewise linear-quadratic} functions, the group Lasso $\ell_1/\ell_2$ norm \cite{ZS17}, and many {\em spectral} matrix functions \cite{CDZ17} including the nuclear norm.  This condition~\eqref{eq:Char} appears recently in \cite[Theorem~3.3]{FNP23} to characterize of {\em strong minimum property} at $\ox$. Consequently, the isolated calmness of $S$ at $(\bar b,\bar \mu)$ for $\ox$ is equivalent to the property that $\ox$ is a {\em strong solution} (of order $2$) of problem~\eqref{eq:RLS} at $(\bar b,\bar \mu)$. This is consistent with the results in \cite{AG08,ZT95}, which established the equivalence between the isolated calmness of the inverse of subdifferential mapping of a convex function and the strong minimum property. But our setup is different, especially on the choice of parameters.

When the function $h$ is written as a composite function $g(Kx)$ as in problem~\eqref{eq:MLS}, the characterization \eqref{eq:Char} turns into 
\begin{equation}\label{eq:Char2}
\Ker\Phi\cap K^{-1}\left(T_{\partial g^*(\oy)}(K\ox)\right)=\{0\}
\end{equation}
with any multiplier $\oy\in \partial g(K\ox)$ satisfying $K^*\oy=\ov$, under some mild conditions on the regularizer $g$; see Theorem~\ref{thm:NSC}. With this composite model, another parameterized problem with an extra canonical perturbation $v\in \YY$ is usually considered
\begin{equation}\label{eq:CLS}
\min_{x\in \XX}\quad \frac{1}{2\mu}\|\Phi x-b\|^2+g(Kx+v),    
\end{equation}
that allows a small noise in computing $Kx$. The corresponding  {\em primal-dual}/Lagrange system that involves both the optimal solution $x$ and Lagrange multiplier $y$ is usually written by
\begin{equation}\label{eq:KKT}
\begin{pmatrix}0\\v\end{pmatrix}\in\begin{pmatrix} 0&K^*\\-K&0\end{pmatrix}\begin{pmatrix}x\\y\end{pmatrix}+ \begin{pmatrix}\frac{1}{\mu}\Phi^*(\Phi x-b)\\\partial g^*(y)\end{pmatrix}.
\end{equation}
This particular system is close with those in \cite{BM24,CS18,DR97,DSZ17,LPB21,MMS22,ZZ16} for constrained and composite optimization problems with the additional tilt (linear) perturbation. Again, the main difference is at the perturbation on $(b,\mu)$ and the general setting of function $g$ in our paper. We show that condition~\eqref{eq:Char2} is still a characterization for the isolated calmness of the solution mapping of system~\eqref{eq:KKT} at $(\bar b,\bar \mu,0)$ for $(\ox,\oy)$
 under some element conditions including the following {\em Strict Robinson's Constraint Qualification} (SRCQ, in brief) for the composite model $g(Kx)$
 \begin{equation}\label{con:SRCQQ}
     \Ker K^*\cap T_{\partial g(K\ox)}(\oy)=\{0\};
 \end{equation}
 see our Corollary~\ref{cor:Eqi3}.
The involvement of SRCQ is consistent with the theory built up for the isolated calmness of KKT/Lagrange systems for constrained and composite optimization problems in \cite{CS18,DSZ17,DR97,ZZ16}. A crucial part in our result is the assumption that the function $g$ satisfies the quadratic growth condition at $K\ox$ for $\oy$ and that its Fenchel conjugate $g^*$ also satisfies the quadratic growth condition at $\oy$ for $K\ox$. The class of these functions are broad in optimization including  convex piecewise linear-quadratic functions, the group Lasso $\ell_1/\ell_2$ norm, and many {\em spectral} matrix functions aforementioned.

Actually, our paper also studies the more general parametrized optimization problem
\begin{equation}\label{eq:GP}
    \min_{x\in \XX}\quad f(x,p)+g(Kx+v),
\end{equation}
where $\mathbb{P}$ is another Euclidean (parameter) space and $f:\XX\times \mathbb{P}\to \R$ is a continuously twice-differentiable function, at which $f(\cdot,p)$ is convex for any $p\in \mathbb{P}$; see our Theorem~\ref{thm:calm}. But our main results with full characterizations for isolated calmness will be focused on the least-squares regularized problem~\eqref{eq:RLS} in Section~4. Due to the convexity assumption on our functions, the closest model to \eqref{eq:GP} among other papers on isolated calmness of KKT/Lagrange systems \cite{BM24,CS18,DR97,DSZ17,LPB21,MMS22,ZZ16} is \cite{CS18} proposed by Cui and Sun with some specific function $g$ involving the nuclear norm and polyhedral constraints that we will discuss more in Section~4. Our conditions returns those in \cite{CS18} and our results can  recover some important part of their main one \cite[Theorem~5.3]{CS18}, but most importantly they can extend it to more general cases of $g$.   Another significant component of  our paper is the study on formulating the {\em kernel} of the {\em subgradient graphical derivative} of convex functions and establishing the so-called {\em zero-product property} for it in Section~3. These results together with the Levy-Rockafellar criterion play crucial roles in our analysis of characterizing the isolated calmness of solution mappings of variational systems in \eqref{eq:OS} and \eqref{eq:KKT}.

\section{Preliminaries}
\setcounter{equation}{0}
Throughout this paper, we suppose that $\XX$ is a Euclidean finite dimensional space with  the inner product $\la \cdot,\cdot \ra$ that endorses the Euclidean  norm $\|\cdot\|$. $\B_\ve(x)$ stands for the open ball in $\XX$ with radius $\ve>0$ and center $x\in \XX$. Most of definitions and notations in this section come from several classical monographs on Convex Analysis \cite{R70} and Variational Analysis \cite{BS00, I17, M1,RW98}. The most important geometric structure used in our paper is the (Bouligand/Severi) {\em tangent/contingent cone} to a closed set $\Omega\subset \XX$ at a point $\ox\in \Omega$ defined by
\begin{equation}\label{eq:Tg}
    T_\Omega(\ox):=\{w\in \XX|\; \exists t_k\dn 0, w_k\to w, \ox+t_k w_k\in \Omega\}=\Limsup_{t\dn 0}\frac{\Omega-\ox}{t}. 
\end{equation}
When $\Omega$ is a convex set, the tangent cone $T_\Omega(\ox)$ is computed by
\begin{equation}\label{eq:TgCon}
T_\Omega(\ox)=\cl[\cone(\Omega-\ox)],
\end{equation}
where $\cone\Omega:=\{tx|\, t\in \R_+, x\in \Omega\}$ is the conic hull of $\Omega$ and $\cl\O$ is the closure of $\O$. The {\em relative interior} of $\O$ is defined by
\begin{equation}\label{eq:ri}
\ri\O:=\{x\in \O|\, \exists \ve>0,\B_\ve(x)\cap {\rm aff}\, \O\subset \O\},
\end{equation}
where ${\rm aff}\, \O$ is the {\em affine hull} of $\O$. We denote $d(x;\Omega)$ for the distance from $x\in \XX$ to the set $\Omega$
\[
d(x;\Omega)=\inf\,\{\|x-u\||\, u\in \O\}. 
\]
Given a set-valued mapping $F:\XX\tto \YY$ between two Euclidean spaces $\XX$ and $\YY$, the domain and the graph of $F$ are defined respectively by
\[
   \dom F:=\{x\in \XX|\; F(x)\neq \emptyset\}\quad\mbox{and}\quad \gph F:=\{(x,y)\in \XX\times \YY|\, y\in F(x)\}.   
\]
The {\em kernel} of $F$ is 
\[
\Ker F:=F^{-1}(0)=\{x\in \XX|\, 0\in F(x)\}.
\]
 With $(\ox,\oy)\in \gph F$, $F$ is said to be {\em metrically subregular} at $\ox$ for $\oy$ if there exist $\kk,\ve>0$ such that 
\begin{equation}\label{def:MS}
    d(x; F^{-1}(\oy))\le \kk d(\oy; F(x))\quad\mbox{for}\quad x\in \B_\ve(\ox). 
\end{equation}
If, additionally, $\ox$ is an isolated point of $F^{-1}(\oy)$, we say $F$ to be {\em strongly metrically subregular} at $\ox$ for $\oy$.  The most important definition in our paper is the so-called {\em isolated calmness} for set-valued mappings. A set-valued mapping $S:\YY\tto \XX$ is said to be {\em isolated calm} at $\oy\in \dom S$ for $\ox\in S(\oy)$ whenever $F=S^{-1}$ is strongly metrically subregular at $\ox$ for $\oy$. Specifically, $S$ is isolated calm at $\oy$ for $\ox$ if there exist $\ell>0$ and a neighborhood $U\times V\subset \XX\times \YY$ of $(\ox,\oy)$ such that 
\begin{equation*}
    \|x-\ox\|\le \ell \|y-\oy\|\quad \mbox{when}\quad x\in S(y)\cap U\;\mbox{and}\; y\in V,
\end{equation*}
or equivalently
\begin{equation}\label{eq:IC}
    S(y)\cap U\subset \ox+\ell\|y-\oy\|\B\quad \mbox{for all}\quad y\in V,
\end{equation}
where $\B$ is the unit ball on $\XX$.  The reader may find more about the history of metric subregularity, isolated calmness, and their applications to optimization and many variational systems in \cite{DR09,I17,KK02}. 

It is well-known that the strong metric subregularity can be fully characterized via the so-called {\em graphical derivative} \cite[Definition~8.33]{RW98} of set-valued mapping $F$ at $\ox$ for $\oy$, which is another set-valued mapping $DF:\XX\tto \YY$  defined by
\begin{equation}\label{def:GD}
DF(\ox|\,\oy)(w):=\{z\in \YY|\, (w,z)\in T_{{\rm gph}\, F }(\ox,\oy)\}\quad \mbox{for any}\quad w\in \XX.
\end{equation}
Specifically, when the graph of $F$ is locally closed around $(\ox,\oy)\in \gph F$, $F$ is strongly metrically subregular at $\ox$ for $\oy$ if and only if 
\begin{equation}\label{eq:LR0}
DF(\ox|\,\oy)^{-1}(0)=\{0\}\quad \mbox{or equivalently}\quad \Ker DF(\ox|\,\oy)=\{0\},
\end{equation}
which is usually referred as Levy-Rockafellar criterion for strong metric subregularity; see, e.g., \cite[Theorem~4C.1]{DR09} and \cite[Proposition~4.1]{L96}.  Equivalently, when the graph of $S$  is locally closed around $(\oy,\ox)\in \gph S$, $S$ has the isolated calmness property at $\oy$ for $\ox$ if and only if 
\begin{equation}\label{eq:LR02}
    DS(\oy|\,\ox)(0)=\{0\}.
\end{equation}

 A significant case of set-valued mapping is the subdifferential mapping $\partial \varphi:\XX\tto \XX$ of a proper l.s.c. convex function $\varphi:\XX\to \oR:=\R\cup\{\infty\}$ defined by
 \[
 \partial \varphi(x):=\{v\in \XX|\, \varphi(u)-\varphi(x)\ge \la v, u-x\ra\ , \forall u\in \XX\}.
 \]
The Fenchel {\em conjugate} of $\varphi$ is known as
\[
\varphi^*(v):=\sup\{\la v,x\ra-\varphi(x)|\, x\in \XX\}\quad\mbox{for}\quad v\in \XX.
\] 
It is well-known \cite{AG08,ZT95} that $\partial\varphi$ is metrically subregular at $\ox\in \dom \varphi$ for $\ov\in \partial \varphi(\ox)$ if and only if $\varphi$ satisfies the {\em quadratic growth condition} at $\ox$ for $\ov$ in the sense that there exist $c,\ve>0$ such that 
\begin{equation}\label{eq:QGC}
    \varphi(x)\ge \varphi(\ox)+\la \ov, x-\ox\ra+\frac{c}{2}\left[d(x;(\partial\varphi)^{-1}(\ov))\right]^2\quad \mbox{for all}\quad x\in \B_\ve(\ox).
\end{equation}
Moreover, $\partial \varphi$ is strongly metrically subregular at $\ox$ for $\ov$ (or  $(\partial\varphi)^{-1}$ is isolated calm at $\ov$ for $\ox$)  if and only if there exist $c,\ve>0$ such that
\begin{equation}\label{eq:SM}
    \varphi(x)\ge \varphi(\ox)+\la \ov, x-\ox\ra+\frac{c}{2}\|x-\ox\|^2\quad \mbox{for all}\quad x\in \B_\ve(\ox),
\end{equation}
which can be characterized by \cite[Theorem~13.24]{RW98}
\begin{equation}\label{eq:SS}
    d^2\varphi(\ox|\, \ov)(w)>0\quad \mbox{for all}\quad w\neq 0, 
\end{equation}
where $d^2\varphi(\ox|\, \ov):\XX\to \oR$ is the {\em second subderivative}  defined by
\begin{equation}\label{def:SS}
    d^2\varphi(\ox|\, \ov)(w):=\liminf_{w^\prime\to w,t\dn 0}\dfrac{\varphi(\ox+tw^\prime)-\varphi(\ox)-t\la \ov,w^\prime\ra}{\frac{1}{2}t^2}.
\end{equation}
The function $\varphi$ is said to be {\em twice epi-differentiable} at $\ox$ for $\ov$ if for any $w\in \XX$ and $t_k\dn 0$, there exists $w_k\to w$ such that the limit 
\[
\lim_{k\to \infty}\dfrac{\varphi(\ox+t_kw_k)-\varphi(\ox)-t_k\la \ov,w_k\ra}{\frac{1}{2}t_k^2}
\]
exists and equals to $d^2\varphi(\ox|\, \ov)(w)$ defined in \eqref{def:SS}. The class of twice epi-differentiable functions is  broad including many important functions in optimization theory such as  convex piecewise linear-quadratic functions, {\em full amenable} functions, and many {\em composite and spectral} functions \cite{MS20,RW98}.

The key second-order structure used in our paper is the {\em subgradient graphical derivative} $D\partial \varphi(\ox|\,\ov):\XX\tto \XX$, which is the graphical derivative of subgradient mapping $\partial\varphi$ at $\ox$ for $\ov\in \partial \varphi(\ox)$. Specifically, 
\[
D\partial \varphi(\ox|\,\ov)(w)=\left\{z\in \XX|\, (w,z)\in T_{{\rm gph}\, \partial \varphi}(\ox, \ov)\right\}.
\]
When $\varphi$ is convex, $\partial \varphi $ is {\em monotone}. This easily leads to the {\em positive semi-definite property} of $D\partial\varphi(\ox|\,\ov)$ in the sense that 
\begin{equation}\label{eq:SPD0}
    \la z,w\ra\ge 0\quad \mbox{for any}\quad z\in D\partial\varphi(\ox|\, \ov)(w). 
\end{equation}
The equal sign will be investigated later in Section~\ref{sec:K} with the so-called {\em zero-product property} for the subgradient graphical derivative.

Finally, with a closed convex set $\Omega\subset \XX$, the indicator function $\delta_\O(x)$ is equal to $0$ if $x\in \O$ and $\infty$, otherwise. The {\em normal cone} to $\O$ at $\ox\in \O$ is the subdifferential of $\delta_\O$ at $\ox$,  denoted by $N_\O(\ox)$. It is well known that 
\[
N_\O(\ox)=\{v|\, \la v, x-\ox\ra\le 0, \forall x\in \O\}=(T_\O(\ox))^*,
\]
which is the {\em polar dual cone} of the tangent cone $T_\O(\ox)$ in \eqref{eq:TgCon}.

\section{Kernel and zero-product property of the subgradient graphical derivative of convex functions}\label{sec:K}
\setcounter{equation}{0}
Let us start this section with a technical lemma useful for later analysis. It has some root in \cite[Proposition~4.1]{GO16}. For the sake of completeness, we provide the full proof here. 

\begin{Lemma}[Kernel of graphical derivative]\label{lem:Ker} Let $F:\XX\tto\YY$ be a set-valued mapping whose graph is locally closed around $(\ox,\oy)\in \gph F$. Then we have
\begin{equation}\label{eq:KerDS}
    \Ker D F(\ox|\,\oy)\supset  T_{F^{-1}(\oy)}(\ox). 
\end{equation} 
Moreover, if  $F$ is metrically subregular at $\ox$ for $\oy$ with some modulus $\kk>0$, then  
\begin{equation}\label{eq:KerD}
    \Ker D F(\ox|\,\oy)=T_{F^{-1}(\oy)}(\ox).
\end{equation} 
Consequently, if the proper l.s.c. convex function $\varphi:\XX\to \oR$ satisfies the quadratic growth condition \eqref{eq:QGC} at $\ox\in \dom \varphi$ for $\ov\in \partial \varphi(\ox)$, then 
\begin{equation}\label{eq:Ker}
   \Ker D\partial \varphi(\ox|\,\ov)=T_{\partial \varphi^*(\ov)}(\ox).
\end{equation}
\end{Lemma}
\begin{proof} 
To justify  \eqref{eq:KerDS}, pick any $w\in T_{F^{-1}(\oy)}(\ox)$. There exist $t_k\dn 0$ and $w_k\st{\XX}\to w$ such that $\ox+t_kw_k\in F^{-1}(\oy)$, which means $\oy\in F(\ox+t_kw_k)$. This leads us to $0\in DF(\ox|\,\oy)(w)$, i.e., $w\in \Ker D F(\ox|\,\oy)$. Formula \eqref{eq:KerDS} is verified.

Next, suppose that $F$ is metrically subregular at $\ox$ for $\oy$. To verify \eqref{eq:KerD}, we just need to show the inclusion ``$\subset$'' in \eqref{eq:KerD}. Indeed, for any $w\in \Ker DF(\ox|\,\oy)$, i.e.,  $0\in DF(\ox|\,\oy)(w)$, we find a sequence $t_k\dn 0$ and $(w_k,z_k)\st{\XX\times \YY}\to (w,0)$ such that $\oy+t_kz_k\in F(\ox+t_kw_k)$. For sufficiently large $k$, we obtain from \eqref{def:MS} that 
\[
\kk t_k\|z_k\|\ge \kk d(\oy; F(\ox+t_k w_k))\ge d(\ox+t_kw_k;F^{-1}(\oy))=t_k d\left(w_k;\frac{1}{t_k}(F^{-1}(\oy)-\ox)\right).
\]
This allows us to find $w_k^\prime \in \frac{1}{t_k}(F^{-1}(\oy)-\ox)$ such that $\kk\|z_k\|\ge \|w_k-w_k^\prime\|$. As $z_k\to 0$,  we obtain that  $w_k^\prime \to w$. This implies that $w\in T_{F^{-1}(\oy)}(\ox)$ and verifies the inclusion ``$\subset$'' in \eqref{eq:KerD}.

Finally, when the convex function $\varphi$ satisfies the quadratic growth condition~\eqref{eq:QGC} at $\ox\in \dom \varphi$ for $\ov\in \partial \varphi(\ox)$, the subdifferential mapping $\partial \varphi$ is metrically subregular at $\ox$ for $\ov$ as discussed before \eqref{eq:QGC}. By \eqref{eq:KerD}, we have 
\[
\Ker D\partial \varphi(\ox|\,\ov)=T_{\partial \varphi^{-1}(\ov)}(\ox),
\]
which is exactly \eqref{eq:Ker} due to the fact that $\partial\varphi^{-1}(\ov)=\partial \varphi^*(\ov)$. 
 \end{proof}

To proceed, let us recall the following useful lemma from \cite[Lemma~2.4]{N24}.

\begin{Lemma}\label{lem:dual} Let $\varphi:\XX\to \oR$ be a proper convex function and $(\ox,\ov)\in \gph \partial \varphi$. For any $z\in D\partial \varphi(\ox|\,\ov)(w)$, we have
\begin{equation}\label{eq:InDC}
    2\la z,w\ra\ge d^2\varphi(\ox|\,\ov)(w)+d^2\varphi^*(\ov|\,\ox)(z).
\end{equation}
\end{Lemma}

This lemma allows us to show next that the kernels of the subgradient graphical derivative $D\partial \varphi(\ox|\,\ov)$ and the second subderivative  $d^2\varphi(\ox|\,\ov)$ are the same.

\begin{Theorem}[Kernels of the subgradient graphical derivative and the second subderivative]\label{thm:Ker} Let $\varphi:\XX\to \oR$ be a proper convex function and $(\ox,\ov)\in \gph \partial \varphi$. Then we have
\begin{equation}\label{eq:KerS}
   D\partial \varphi^*(\ov|\,\ox)(0)=\Ker D\partial \varphi(\ox|\,\ov)=\Ker d^2\varphi(\ox|\, \ov).
\end{equation}
Consequently, if $\varphi$ satisfies the quadratic growth condition at $\ox$ for $\ov$, then   
\begin{equation}\label{eq:KerSS}
    \Ker d^2\varphi(\ox|\, \ov)=T_{\partial \varphi^*(\ov)}(\ox).
\end{equation}

\end{Theorem}
\begin{proof}
The first equality in \eqref{eq:KerS} is trivial due to the fact that $\partial \varphi^*=(\partial \varphi)^{-1}$ and that
\[
z\in D\partial\varphi(\ox|\,\ov)(w)\quad \mbox{if and only if}\quad w\in D(\partial \varphi)^{-1}(\ov|\,\ox)(z).
\]
Next we prove the implication ``$\subset$'' for the second equality in \eqref{eq:KerS}. Pick any $w\in \Ker D\partial \varphi(\ox|\,\ov)$, i.e., $0\in D\partial \varphi(\ox|\,\ov)(w)$. By Lemma~\ref{lem:dual}, we have 
\[
0\ge d^2\varphi(\ox|\, \ov)(w)+d^2\varphi^*(\ov|\, \ox)(0).
\]
Since both $\varphi$ and $\varphi^*$ are convex functions, we have   $d^2\varphi(\ox|\,\ov)(w),d^2\varphi^*(\ov|\,\ox)(0)\ge 0$. This implies that  
\begin{equation}\label{eq:hh*}
d^2\varphi(\ox|\,\ov)(w)= 0.
\end{equation}
which clearly verifies the implication ``$\subset$'' in \eqref{eq:KerS}. To ensure the converse implication ``$\supset$'' in \eqref{eq:KerS}, pick any $w\in \Ker d^2\varphi(\ox|\, \ov)$ satisfying \eqref{eq:hh*}.  Thus, there exist $t_k\dn 0$ and $w_k\to w$ such that 
\[
\varphi(\ox+t_kw_k)-\varphi(\ox)-t_k\la \ov,w_k\ra=o(t_k^2). 
\]
Since $\ov\in \partial \varphi(\ox)$, the latter gives us that 
\[
\varphi(\ox+t_kw_k)-\la \ov, \ox+t_kw_k\ra=\varphi(\ox)-\la \ov,\ox\ra+o(t_k^2)=\inf_{x\in \XX} \varphi(x)-\la \ov,x\ra+o(t_k^2). 
\]
By the classical Br{\o}ndsted-Rockafellar variational principle \cite{BR65}, there exists some $u_k\in \XX$ such that 
\begin{equation}\label{eq:Eke}
    \|u_k-\ox-t_kw_k\|\le \sqrt{o(t^2_k)}\quad \mbox{and}\quad d(\ov;\partial \varphi(u_k))\le \sqrt{o(t^2_k)}.
\end{equation}
We may write $u_k=\ox+t_kw_k^\prime$ with some $w_k^\prime \in \XX$ satisfying $\|w^\prime_k-w_k\|\to 0$ and find some $z_k\to 0$ such that $\ov+t_kz_k\in \partial \varphi(u_k)=\partial \varphi(\ox+t_kw_k^\prime)$. It follows that $0\in D\partial \varphi(\ox|\,\ov)(w)$, i.e., $w\in \Ker D\partial \varphi(\ox|\,\ov). $  This verifies \eqref{eq:KerS}. 

Finally, suppose that $\varphi$ satisfies the quadratic growth condition at $\ox$ for $\ov$. Then $\partial \varphi$ is metrically subregular at $\ox$ for $\ov$. Formula \eqref{eq:KerSS} follows from \eqref{eq:KerD} with $F=\partial \varphi$ and \eqref{eq:KerS}.
\end{proof}

Note that the formula \eqref{eq:hh*} was obtained in \cite[Lemma~3.2]{FNP23} with a different proof. Before establishing the main result of this section, the zero-product property of subgradient graphical derivative for convex functions, we introduce the so-called {\em primal-dual} quadratic growth condition for convex functions as follows.   

\begin{Definition}[Primal-dual quadratic growth condition] Let $\varphi:\XX\to \oR$ be a proper l.s.c. convex function with $\ox\in \dom \varphi$ and $\ov\in \partial \varphi(\ox)$. We say $\varphi$ to satisfy the primal-dual quadratic growth condition at $\ox$ for $\ov$ if it satisfies the quadratic growth condition at $\ox$ for $\ov$ and its Fenchel conjugate $\varphi^*$ also satisfies the quadratic growth condition at $\ov$ for $\ox$.
\end{Definition}

\begin{Remark}\label{rem:QL}{\rm  

Let us provide here a few important convex functions $\varphi$ in optimization that satisfy the primal-dual quadratic growth condition at $\ox\in \dom \varphi$ for any $\ov\in \partial \varphi(\ox)$. As discussed before \eqref{eq:QGC}, the function $\varphi$  satisfies the primal-dual quadratic growth condition at $\ox$ for $\ov$ if and only if $\partial \varphi$ is metrically subregular at $\ox$ for $\ov$ and $\partial\varphi^*$ metrically subregular at $\ov$ for $\ox$. 

(a) (Convex piecewise linear-quadratic functions \cite[Definition~10.20]{RW98}). The proper l.s.c. convex function $\varphi:\R^n\to \oR$ is called  {\em convex piecewise linear-quadratic} if $\dom \varphi$ is a union of finitely many polyhedral sets, for  each of which the function $\varphi(x)$ has a quadratic expression $\frac{1}{2}\la Ax,x\ra+\la b,x\ra+c$ for some scalar $c\in \R$, vector $b\in \R^n$, and positive semi-definite $n\times n$ matrix $A$.  \cite[Theorem~11.14 and Proposition~12.30]{RW98} tell us that both $\partial \varphi$ and $\partial \varphi^*$ are {\em polyhedral} mappings in the sense that their graphs are unions of finitely many  polyhedral convex sets in $\R^n\times \R^n.$ Moreover, it is well-known that the inverse of any polyhedral mapping is metrically subregular at any point on its graph (see, e.g., \cite[Proposition~3H.1]{DR09}). Hence, both $\partial \varphi$ and $\partial \varphi^*$ are metrically subregular at any point on their graphs, which means $\varphi$ satisfies the primal-dual quadratic growth condition at any $\ox\in \dom \varphi$ for any $\ov\in \partial \varphi(\ox)$. 


(b) (The analysis group Lasso regularizer) The $\ell_1/\ell_2$ norm function $\varphi(x)=\|x\|_{1,2}$, $x\in \R^n$ (a.k.a. the analysis group Lasso regularizer) is defined by
\begin{equation}\label{eq:GLasso}
\|x\|_{1,2}:=\sum_{J\in \mathcal{J}}\|x_J\|\quad \mbox{for any}\quad x\in \R^n.
\end{equation}
Here, the index set $\mathcal{J}$ represents for a partition of $\{1,2,\ldots,n\}$. For each $J\in \mathcal{J}$, $x_J\in \R^{|J|}$ is the  vector containing all components $x_j$, $j\in J$ in $x$. \cite[Proposition 9]{ZS17} tells us that $\partial\|\cdot\|_{1,2}$ is metrically subregular at any $\ox\in \R^n$ for any $\ov\in \partial\|\cdot\|_{1,2}(\ox)$. Thus,  $\|\cdot\|_{1,2}$ satisfies the quadratic growth condition at $\ox$ for $\ov$. Next, we claim that $\varphi^*$ also satisfies the quadratic growth condition at $\ov$ for $\ox\in \partial \|\cdot\|^*_{1,2}(\ov)$ by indicating some $\kappa>0$ such that for any $v\in \R^n$
\begin{equation}\label{ine:f}
 \varphi^*(v)\ge\varphi^*(\ov)+\la\ox,v-\ov \ra + \kappa \left[d(v;\partial \varphi (\ox))\right]^2.
\end{equation}
Note that the Fenchel conjugate of $\ell_1/\ell_2$ norm is the  indicator function to 
\[
\B_{\infty,2}:=\{v\in \R^n\,|\, \|v_J\|\le 1\, , J\in \mathcal{J} \}
\]
and  the subdifferential of the $\ell_1/\ell_2$ norm at $\ox$ is
\begin{equation}\label{sub:l1l2}
 \partial \|\cdot\|_{1,2}(\ox)= \left(\prod_{J\in \mathcal{I}}\frac{\ox_J}{\|\ox_J\|}\times\prod_{J\in \mathcal{I}^c}\B_J \right),
\end{equation}
where 
\[
\mathcal{I}:=\{J\in \mathcal{J}\,|\, \ox_J \neq 0\},\quad \mathcal{I}^c:=\mathcal{J}\backslash\mathcal{I}\quad \mbox{and}\quad \B_J:=\left\{x\in \R^{|J|}|\, \|x\|\le 1\right\}.
\]
If $v\notin \B_{\infty,2}$,  \eqref{ine:f} is trivially true, as $\varphi^*(v)=\infty$. It suffices to consider $v \in \B_{\infty,2}$. As $\ov\in \partial \|\cdot\|_{1,2}(\ox)$, it follows from \eqref{sub:l1l2} that 
\[
\ov_J= \frac{\ox_J}{\|\ox_J\|}\quad \forall J\in \mathcal{I} \quad \mbox{and} \quad \|\ov_J\|\le1 \quad \forall J\in \mathcal{I}^c.
\]
By \eqref{sub:l1l2}, inequality \eqref{ine:f} is equivalent to
\begin{equation}\label{ine:f2}
   0\ge  \la \ox,v-\ov \ra+\kappa \sum_{J\in \mathcal{I}}\|v_J-\ov_J\|^2=\sum_{J\in \mathcal{I}}\left(\la \ox_J,v_J-\ov_J \ra+\kappa \|v_J-\ov_J\|^2\right)
\end{equation}
Note that 
\[
1\ge \|v_J\|^2=\|v_J-\ov_J+\ov_J\|^2=\|v_J-\ov_J\|^2+2\la v_J-\ov_J,\ov_J\ra+\|\ov_J\|^2, 
\]
which yields 
\begin{equation}\label{eq:tri}
   \frac{1}{\|\ox_J\|}\la \ox_J,v_J-\ov_J\ra=  \la \ov_J,v_J-\ov_J\ra \le-\frac{1}{2}\|v_J-\ov_J\|^2. 
\end{equation}
Define $\kk:=\frac{1}{2}\min\{\|\ox_J\||\, J\in \mathcal{I}\}$. Then we obtain from the latter that 
\[
\la \ox_J,v_J-\ov_J\ra\le -\kk \|v_J-\ov_J\|^2\quad \mbox{for any}\quad J\in \mathcal{I},
\]
which clearly verifies \eqref{ine:f2} and thus \eqref{ine:f}.

(c) (Convex spectral functions \cite[Definition~5.2]{LS05a}) A spectral function $\varphi$ is an extended-real-value function defined on $\R^{m\times n}$ ($m\le n$) of the form $f\circ \sigma(X)$ for $X\in\R^{m\times n} $ where  $\sigma(X)=(\sigma_1(X),\sigma_2(X),\ldots,\sigma_m(X))^T$ of {\em singular values} of $X$: $\sigma_1(X)\ge\sigma_2(X)\ge\ldots\ge\sigma_m(X)\ge 0$ from its {\em singular value decomposition}
\begin{equation}\label{eq:SVD}
    X=U\begin{pmatrix}\sigma_1(X)&\ldots&0&0&\ldots& 0\\
0&\ddots&0&0&\ldots& 0\\ 
0&\ldots &\sigma_{m}(X)&0&\ldots 
&0\end{pmatrix}_{m\times n}V^T
\end{equation}
with orthogonal matrices $U\in \R^{m\times m}$ and $V\in \R^{n\times n}$, and $f:\R^m \to \oR$ is a l.s.c. convex function that is {\em absolutely symmetric}  in the sense that  
\[
f(x)=f(|x_{i_1}|,\ldots, |x_{i_m}|)\quad \mbox{for any}\quad x\in \R^m
\]
with any permutation $(i_1,\ldots,i_m)$ of $(1, \ldots,m)$. With $\varphi(X)=f(\sigma(X))$, it is well-known that $\varphi^*=f^*\circ \sigma$. Given a pair $(\OX,\OY)\in \gph \partial \varphi$, it follows from \cite[Proposition~3.8]{CDZ17} that $\partial \varphi$ satisfies the primal-dual quadratic growth condition at $\OX$ for $\OY$, if $f$ satisfies the primal-dual growth condition at $\sigma(\OX)$ for $\sigma(\OY)$. A particular case is that $f$ is a convex piecewise linear-quadratic function as discussed in (a). An important spectral function is the nuclear norm, which is usually used in optimization to promote the low-rank property for optimal solutions, defined by
\[
\|X\|_*:=\sigma_1(X)+\ldots+\sigma_m(X)=\|\sigma(X)\|_1.
\]
In this case, the outer function $f(x)$ is the $\ell_1$ norm, which is a convex piecewise linear function. Thus, the nuclear norm satisfies the primal-dual quadratic growth condition at any point on its graph. \endproof

}
\end{Remark}

\begin{Theorem}[Zero-product property of subgradient graphical derivative]\label{thm:Zero} Let $\varphi:\XX\to \oR$ be a proper convex function and $(\ox,\ov)\in \gph \partial \varphi$. For any $z\in D\partial  \varphi(\ox|\, \ov)(w)$, we have
\begin{equation}\label{eq:ZeroSD}
\la z,w\ra=0\quad \Longrightarrow\quad  z\in  D\partial  \varphi(\ox|\, \ov)(0)\;\;\mbox{and}\;\;w\in D\partial  \varphi^*(\ov|\,\ox)(0).
\end{equation} 
The opposite implication holds when one of the following conditions is satisfied:
\begin{itemize}
 \item [{\bf (i)}] $ \varphi$ is twice epi-differentiable at $\ox$ for $\ov$,
 \item [{\bf (ii)}] $ \varphi$ satisfies the primal-dual quadratic growth condition  at $\ox$ for $\ov$.
\end{itemize}
Consequently, if $ \varphi$ satisfies the primal-dual quadratic growth condition at $\ox$ for $\ov$,  then for any $z\in D\partial  \varphi(\ox|\ov)(w)$ we have
\begin{equation}\label{eq:Zero}
\la z,w\ra=0\quad \mbox{if and only if }\quad  z\in T_{\partial  \varphi(\ox)}(\ov)\;\mbox{and}\;w\in T_{\partial  \varphi^*(\ov)}(\ox).
\end{equation} 
\end{Theorem}
\begin{proof}
Let us prove the implication ``$\Rightarrow$'' in \eqref{eq:ZeroSD} by taking any pair $(z,w)\in \XX\times \XX$ satisfying $z\in D\partial\varphi(\ox|\,\ov)(w)$ and  $\la z,w\ra=0$. By Lemma~\ref{lem:dual}, we have 
\[
0\ge d^2\varphi(\ox|\ov)(w)+d^2\varphi^*(\ov|\ox)(z).
\]
Since both $\varphi$ and $\varphi^*$ are convex functions, we have $d^2\varphi(\ox|\ov)(w), d^2\varphi^*(\ov|\ox)(z)\ge 0$. It follows that $w\in \Ker d^2\varphi(\ox|\ov)$  and $z\in \Ker d^2\varphi^*(\ov|\ox)$. Applying Theorem~\ref{thm:Ker} for convex functions $\varphi$ and $\varphi^*$ gives us \eqref{eq:ZeroSD} immediately.

Next, let us verify the converse ``$\Leftarrow$'' implication in \eqref{eq:ZeroSD} when either {\bf (i)} or {\bf (ii)} holds. Pick any pair $(z,w)\in \XX\times \XX$ satisfying $z\in D\partial \varphi(\ox|\,\ov)(w)$ with  $z\in D\partial \varphi(\ox|\,\ov)(0)$ and $w\in D\partial \varphi^*(\ov|\,\ox)(0)$. 

{\bf Case I:} $\varphi$ is twice epi-differentiable at $\ox$ for $\ov$. In this case, we know from \cite[Lemma~3.6]{CHNT21} that 
\[
\la z,w\ra=d^2 \varphi(\ox|\,\ov)(w)\quad \mbox{for any}\quad z\in D\partial \varphi(\ox|\,\ov)(w),
\]
which clearly implies that $\la z,w\ra=0$ due to formula \eqref{eq:KerS} that gives $w\in \Ker d^2 \varphi(\ox|\,\ov)$.

{\bf Case II:} $\varphi$ satisfies the primal-dual quadratic growth condition at $\ox$ for $\ov$. As $w\in \Ker D\partial \varphi(\ox|\,\ov)$ and $z\in \Ker D\partial \varphi^*(\ov|\,\ox)$, it follows from Lemma~\ref{lem:Ker} that 
\begin{equation}\label{eq:zwT}
z\in T_{\partial \varphi(\ox)}(\ov)\quad \mbox{and}\quad  w\in T_{\partial \varphi^*(\ov)}(\ox).
\end{equation}
By the definition of tangent cone \eqref{eq:Tg}, we find sequences $r_k,s_k\dn 0$ and $(w_k,z_k)\to (w,z)$ such that $\ox+r_kw_k\in \partial \varphi^*(\ov)$ and $\ov+s_kz_k\in \partial \varphi (\ox)$. 
This leads us to
\begin{equation*}
\la \ov,\ox+r_kw_k\ra-\varphi(\ox+r_kw_k)=\varphi^*(\ov)=\la \ov, \ox\ra-\varphi(\ox),
\end{equation*}
which implies that $\varphi(\ox+r_kw_k)-\varphi(\ox)=r_k\la \ov,w_k\ra$.
Since $\ov+s_kz_k\in \partial\varphi(\ox)$, we have
\[
r_k\la \ov,w_k\ra=\varphi(\ox+r_kw_k)-\varphi(\ox)\ge \la \ov+s_kz_k, \ox+r_kw_k-\ox\ra=r_k\la \ov+s_kz_k,w_k\ra.
\]
This gives us that $r_ks_k\la z_k,w_k\ra\le 0$, which means $\la z_k,w_k\ra\le 0$. Letting $k\to \infty$ yields $\la z,w\ra\le 0$. Moreover, $\la z,w\ra\ge0$ by \eqref{eq:SPD0}. We have $\la z,w\ra=0$. This verifies \eqref{eq:ZeroSD} as well as the equivalence in \eqref{eq:Zero} by Lemma~\ref{lem:Ker}.  The proof is complete. 
\end{proof}

\section{Isolated Calmness for Variational Convex Systems}
\setcounter{equation}{0}
Let us start to consider the following parametric optimization problem, the more general version of  \eqref{eq:MLS}, 
\begin{equation}\label{p:PO}
    \min_{x\in \XX}\quad f(x,p)+g(Kx),
\end{equation}
where $p$ is a parametric variable in a Euclidean space $\mathbb{P}$, $K:\XX\to\YY$ is a linear operator between two Euclidean spaces $\XX$ and $\YY$, $f:\XX\times \mathbb{P}\to \R$ is a continuous twice-differentiable function that is  convex with respect to $x$, and $g:\YY\to \R\cup\{\infty\}$ is a proper l.s.c. convex function. The solution mapping of problem~\eqref{p:PO} is denoted by
\begin{equation}\label{eq:Sp0}
    S(p):={\rm argmin}\{f(x,p)+g(Kx)|\; x\in \XX\}. 
\end{equation}
For a fixed $\op\in \mathbb{P}$, suppose that $S(\op)\neq \emptyset$ and $\ox\in S(\op)$. For any $x\in S(p)$, as $f(x,p)$ is continuously differentiable, it follows from the Fermat rule that 
\begin{equation}
0\in \nabla_x f(x,p)+\partial h(x)\quad \mbox{with}\quad h(x):=g(Kx).
\end{equation}
Since we only need to study the solution mapping $S(p)$ around $\ox$ and use the {\em chain rule} for $\partial h(\ox)$ to extract information of $g$, the following constraint qualification is supposed throughout this section at $\ox$:
\begin{equation}\label{eq:RCQ}
0\in {\rm int}\left(K\ox+\Im K-\dom g\right)\quad \mbox{or equivalently}\quad \Im K+T_{{\rm dom}\, g}(K\bar x)=\YY,
\end{equation}
which is known as the {\em Robinson Constraint Qualification} (RCQ) for composite model \cite[Section~3.4]{BS00}. This condition holds trivially if the function $g$ is continuous on $\YY$ with full domain. Under this condition, there exists some neighborhood $U$ of $\ox$ such that 
\begin{equation}\label{eq:CR}
\partial h(x)=K^*\partial g(Kx)\quad \mbox{for all}\quad x\in U,
\end{equation}
which is  the  chain rule for composite convex functions; see, e.g., \cite[Theorem~2.168]{BS00}. It follows that 
\begin{equation}\label{eq:Sp}
S(p)\cap U=\{x\in U|\; 0\in \nabla_x f(x,p)+K^*\partial g(Kx)\}.
\end{equation}
Its corresponding {\em linearized} solution mapping around the point $(\ox,\op)$ in question is known as 
\begin{equation}\label{eq:BSd}
\check S(v):=\{x\in \XX|\; v\in \nabla_x f(\ox,\op)+\nabla_{xx}^2 f(\ox,\op)(x-\ox)+K^*\partial g(Kx)\}.
\end{equation}
It is well-known that $S$ is isolated calm at $\op$ for $\ox$ provided that $\check S$ is isolated calm at $0$ for $\ox$; see, e.g.,  \cite[Theorem~3I.12 and Theorem~3I.13]{DR09} (a.k.a. the {\em implicit mapping theorem} with strong metric subregularity/isolated calmness.) The isolated calmness of $\check S$ can be characterized via the graphical derivative or the second subderivative as in \eqref{eq:LR02}  and \eqref{eq:SS}.
However, computing these structures for the composite function $g(Kx)$ can be very complicated. In the following result, we establish a simple sufficient condition.  

\begin{Theorem}[Sufficient conditions for the isolated calmness of solution mapping]\label{thm:Suff} Suppose that the function $f$ is twice continuously differentiable around $(\ox,\op)$ with $\ox\in S(\op)$. Then $S$ is isolated calm at $\op$ for $\ox$ provided that 
\begin{equation}\label{eq:Kerh}
        \Ker \nabla^2_{xx}f(\ox,\op)\cap \Ker d^2 h(\ox|\ov)=\{0\}\quad \mbox{with}\quad h(x)=g(Kx)\quad \mbox{and}\quad \ov:=-\nabla_x f(\ox,\op).  
\end{equation}
Let $\oy\in \partial g(K\ox)$ with $K^*\oy=-\nabla_x f(\ox,\op)$. If, additionally, $g$ satisfies the quadratic growth condition at $K\ox$ for $\oy$,  then   $S$ is isolated calm at $\op$ for $\ox$ if 
\begin{equation}\label{eq:KerK}
        \Ker \nabla^2_{xx}f(\ox,\op)\cap K^{-1}\left( T_{\partial g^*(\oy)}(K\ox)\right)=\{0\}. 
\end{equation}
\end{Theorem}
\begin{proof} Suppose that \eqref{eq:Kerh} is satisfied. Let us claim that $d^2\Psi(\ox|\,0)(w)>0$ for any $w\neq 0$, where 
\[
\Psi(x):= f(\ox,\op)+\la \nabla_x f(\ox,\op),x-\ox\ra+\frac{1}{2}\la \nabla_{xx}^2 f(\ox,\op)(x-\ox),x-\ox\ra+ h(x)\quad \mbox{for}\quad x\in \XX.
\]
Suppose that there is some $w\in \XX$ such that $d^2\Psi(\ox|\,0)(w)\le0$. Note that
\[
0\ge d^2\Psi(\ox|\,0)(w)=\la \nabla^2_{xx}f(\ox,\op) w,w\ra+d^2 h(\ox|\, \ov)(w).
\]
Since $f$ and $h$ are convex functions, we have $\la \nabla^2_{xx}f(\ox,\op) w,w\ra\ge 0$ and $d^2 h(\ox|\, \ov)(w)\ge 0$. It follows that 
\[
0=\la \nabla^2_{xx}f(\ox,\op) w,w\ra=d^2 h(\ox|\, \ov)(w).
\]
Since $\nabla^2_{xx} f(\ox,\op)$ is positive semi-definite, we have   $w\in \Ker \nabla^2_{xx}f(\ox,\op)\cap \Ker d^2 h(\ox|\, \ov)=\{0\}$ by \eqref{eq:Kerh}. Thus $d^2\Psi(\ox|\,0)(w)>0$ for any $w\neq 0$, i.e., $\Ker d^2\Psi(\ox|\,0)=\{0\}$. By Theorem~\ref{thm:Ker}, $\Ker D\partial \Psi(\ox|\,0)=\{0\}$. Thus, $\check S=(\partial \Psi)^{-1}$ is isolated calm at $0$ for $\ox$ by \eqref{eq:LR02}. By the implicit mapping theorem with strong metric subregularity  \cite[Theorem~3I.13]{DR09}, the solution mapping $S$ is isolated calm at $\op$ for $\ox$. 

To prove the second part of the theorem, let us claim that 
\begin{equation}\label{eq:Sub}
\Ker d^2 h(\ox|\, \ov)\subset K^{-1}\left( T_{\partial g^*(\oy)}(K\ox)\right),
\end{equation}
when $g$ satisfies the quadratic growth condition at $K\ox$ for $\oy$. Indeed, for any $w\in\Ker d^2 h(\ox|\, \ov)$, note that 
 \begin{equation}
\begin{aligned}
0=d^2 h(\ox|\,\ov)(w)&=\liminf_{w^\prime\to w, t\dn 0}\dfrac{h(\ox+tw^\prime)-h(\ox)-t\la \ov, w^\prime\ra}{\frac{1}{2}t^2}\\
    &=\liminf_{w^\prime\to w, t\dn 0}\dfrac{g(K\ox+t Kw^\prime)-g(K\ox)-t\la \oy, K w^\prime\ra}{\frac{1}{2}t^2}\\
&\ge d^2 g(K\ox|\; \oy)(Kw)\ge 0.
\end{aligned}
\end{equation}
It follows from \eqref{eq:KerSS} that 
\[
Kw\in \Ker d^2 g(K\ox|\; \oy)=T_{\partial g^*(\oy)}(K\ox).
\]
This clearly verifies \eqref{eq:Sub}. Consequently, condition~\eqref{eq:KerK} implies \eqref{eq:Kerh}. Thus $S$ is isolated calm at $\op$ for $\ox$ when condition~\eqref{eq:KerK} is satisfied. 
\end{proof}

It is natural to ask if condition \eqref{eq:KerK} is also necessary for the isolated calmness. We obtain some affirmative answers for the parametric least-squares regularized problem \eqref{eq:MLS} discussed in the Introduction
\begin{equation}\label{eq:LeS}
 P(b,\mu):\quad   \min_{x\in \XX}\quad \frac{1}{2\mu}\|\Phi x-b\|^2+g(Kx),
\end{equation}
where $\Phi:\XX\to \YY$ is a linear operator and  $\mu>0$ and $b\in \YY$ are two parameters. The solution mapping of problem~\eqref{eq:LeS} (by ignoring a neighborhood $U$ of $\ox$ in \eqref{eq:Sp}) is written by 
\begin{equation}\label{eq:SLS}
    S_{LS}(b,\mu):=\left\{x\in \XX|\; 0\in \frac{1}{\mu}\Phi^*(\Phi x-b)+K^*\partial g(Kx)\right\}. 
\end{equation}
Before establishing the main result of this section, we prove a useful lemma that provides a formula for the tangent cone to the linear inverse image of a closed convex set.
\begin{Lemma}[Tangent cone to linear inverse image of closed convex sets]\label{lem:tan} Let $K:\XX\to \YY$ be a linear operator and $C\subset \YY$ be a closed convex set such that $C\cap\Im K\neq \emptyset$. Suppose that $x_0\in K^{-1}(C)$. Then we have
\begin{equation}\label{eq:Inv}
    T_{K^{-1}(C)}(x_0)=K^{-1}\left(T_{C\cap{\rm Im}\, K}(Kx_0)\right).
\end{equation}
\end{Lemma}
\begin{proof} Let us start with the proof of the implication ``$\subset$'' in \eqref{eq:Inv}. Indeed, for any $w\in  T_{K^{-1}(C)}(x_0)$, we find  sequences $t_k\dn 0$ and  $w_k\to w$ such that $x_0+t_kw_k\in K^{-1}(C)$. It follows that 
\[
Kx_0+t_kKw_k\in C\cap \Im K,
\]
which implies that $Kw \in T_{C\cap{\rm Im}\, K}(Kx_0)$. This clearly verifies the implication ``$\subset$'' in \eqref{eq:Inv}.

To prove the converse implication, note that 
\[
\R_+[{C\cap{\rm Im}\, K}-Kx_0]\subset K\left(\R_+(K^{-1}(C\cap{\rm Im}\, K)-x_0)\right).
\]
Indeed, for any $w$ belonging the left-hand side, there exists $t>0$ and $u\in K^{-1}(C\cap{\rm Im}\, K)$    such that $w=t(Ku-Kx_0)=K(t(u-x_0))$, which is an element of the right-hand side. Since $C$ is a closed convex set, so is  ${C\cap{\rm Im}\, K}$ and $K^{-1}(C\cap{\rm Im}\, K)$. We obtain from the above inclusion and \eqref{eq:TgCon} that
\begin{equation}\label{imp:tan}
    \begin{aligned}
T_{C\cap{\rm Im}\, K}(Kx_0)=\cl\left(\R_+[{C\cap{\rm Im}\, K}-Kx_0]\right)&\subset \cl\left[K\left(\R_+(K^{-1}(C\cap{\rm Im}\, K)-x_0)\right)\right]\\
&\subset \cl\left[K\left(\cl\left(\R_+(K^{-1}(C\cap{\rm Im}\, K)-x_0)\right)\right)\right]\\
&\subset \cl\left[K\left(T_{K^{-1}(C\cap{\rm Im}\, K)}(x_0)\right)\right]\\
&= \cl\left[K\left(T_{K^{-1}(C)}(x_0)\right)\right].
\end{aligned}
\end{equation}
Note also that $x_0+\Ker K\subset K^{-1}(C)$, which implies that $\Ker K\subset T_{K^{-1}(C)}(x_0) $ and that 
\[
\Ker K+T_{K^{-1}(C)}(x_0)=T_{K^{-1}(C)}(x_0). 
\]
Hence $\Ker K+T_{K^{-1}(C)}(x_0)$ is closed. It follows from \cite[Proposition 3.1] {BM09} that $K\left(T_{K^{-1}(C)}(x_0)\right)$ is a closed set. Finally, combining this with the last inclusion in \eqref{imp:tan}  verifies the implication ``$\supset$'' in \eqref{eq:Inv} and completes the proof of the lemma. 
  \end{proof}

The following theorem, which provides necessary and sufficient conditions for isolated calmness of solution mapping \eqref{eq:SLS}, is the main result in our paper.  

\begin{Theorem}[Necessary and sufficient conditions for isolated calmness]\label{thm:NSC}
Let $\ox\in S_{LS}(\bar b,\bar \mu)$  and $\oy\in \partial g(K\ox)$ with $K^*\oy=-\frac{1}{\bar \mu}\Phi^*(\Phi \bar x-\bar b)$.  If $S_{LS}$ has the isolated calmness property at $(\bar b,\bar \mu)$ for $\ox$, then we have 
\begin{equation}\label{eq:Nes}
    \Ker \Phi\cap K^{-1} \left(T_{\partial g^*(\oy)\cap {\rm Im}\,K}(K\ox)\right)=\{0\}.
\end{equation}
If, additionally, $g$ satisfies the quadratic growth condition at $K\ox$ for $\oy$, then $S_{LS}$ has the isolated calmness property at $(\bar b,\bar \mu)$ for $\ox$ provided that 
\begin{equation}\label{eq:Suf}
\Ker \Phi\cap K^{-1}\left(T_{\partial g^*(\oy)}(K\ox)\right)=\{0\}.
\end{equation}
Consequently, condition \eqref{eq:Suf} is necessary and sufficient for the isolated calmness of $S_{LS}$ at $(\bar b,\bar \mu)$ for $\ox$ if one of the following two qualification conditions is true:

\begin{itemize}
\item[\bf(i)] The set $\partial g^*(\oy)$ is polyhedral.

\item[\bf(ii)] $\Im K\cap\ri (\partial g^*(\oy))\neq \emptyset$, where $\ri (\partial g^*(\oy))$ is the relative interior of the set $\partial g^*(\oy)$ defined in \eqref{eq:ri}.

\end{itemize}
\end{Theorem}
\begin{proof} Let us start by supposing $S_{LS}$ is isolated calm at $(\bar b,\bar \mu)$ for $\ox$. Hence, there exist $\kk>0$ and a neighborhood $U\times V\times M$ of $(\ox,\bar b, \bar \mu)$ in $\XX\times \YY\times\R_{++}$  such that 
\begin{equation}\label{eq:ICLS}
S_{LS}(b,\mu)\cap U\subset \ox+\kk[\|b-\bar b\|+|\mu-\bar\mu|]\B_{\XX} \quad \mbox{for all}\quad (b,\mu)\in V\times M.
\end{equation}
Recall $h(x)=g(Kx)$ for $x\in \XX$ and $\ov:=-\frac{1}{\bar \mu}\Phi^*(\Phi \ox-\bar b)$. We claim that   
\begin{equation}\label{eq:KerTg}
\Ker\Phi\cap T_{\partial h^*(\bar v)}(\ox)=\{0\}. 
\end{equation}
Indeed, pick any $w\in \Ker\Phi\cap T_{\partial h^*(\bar v)}(\ox)$. Hence, there exist $t_k\dn 0$ and $w_k\to w$ such that $\ox+t_kw_k\in \partial h^*(\bar v)$, i.e., $\ov\in \partial h(\ox+t_kw_k)$. Let us define 
\[
\mu_k:=\bar \mu\quad\mbox{and}\quad b_k=\bar b+ t_k\Phi w_k\to \bar b.
\]
Note that 
\[
-\frac{1}{\bar \mu}\Phi^*(\Phi (\ox +t_k w_k)-b_k)=-\frac{1}{\bar \mu}\Phi^*(\Phi \ox -\bar b)=\ov\in \partial h(\ox+t_kw_k).
\]
This tells us that $\ox+t_kw_k\in S_{LS}(b_k,\bar\mu)$. For sufficiently large $k$, it follows from \eqref{eq:ICLS} that 
\[
t_k\|w_k\|=\|\ox+t_kw_k-\ox\|\le \kk\|b_k-\bar b\|=\kk t_k\|\Phi w_k\|,
\]
which yields $\|w_k\|\le \kk\|\Phi w_k\|$. By taking $k\to \infty$, we have $\|w\|\le\kk \|\Phi w\|=0$, as $w\in \Ker \Phi$. This verifies the claim \eqref{eq:KerTg}. Note further that $\partial h^*(\ov)=K^{-1}(\partial g^*(\oy))$ due to the Fenchel-Rockafellar duality theorem \cite[Theorem~31.3]{R70}. 
By Lemma~\ref{lem:tan}, condition~\eqref{eq:KerTg} is exactly \eqref{eq:Nes}. Thus, \eqref{eq:Nes} is a necessary condition of the isolated calmness of $S_{LS}$ at $(\bar b,\bar \mu)$ for $\ox$. 

Define $f(x,p):=\frac{1}{2\mu}\|\Phi x-b\|^2$ with $p=(b,\mu)$ and note that 
\begin{equation}\label{eq:KHes}
\Ker \nabla^2_{xx}f(\ox,\op)=\Ker \frac{1}{\bar \mu} \Phi^*\Phi=\Ker \Phi. 
\end{equation}
Condition~\eqref{eq:Suf} is the same with \eqref{eq:KerK} in this case. Hence, it is sufficient for the isolated calmness of $S_{LS}$ at $\bar p=(\bar b,\bar\mu)$ for $\ox$ by Theorem~\ref{thm:Suff}. 

Finally, it remains to show that \eqref{eq:Nes} and \eqref{eq:Suf} are identical  when either  {\bf(i)} or {\bf(ii)} is satisfied. Indeed, since $\Im K$ is polyhedral, if either {\bf(i)} or {\bf(ii)} holds, we obtain from the formula of tangent cone to intersection sets from Convex Analysis that   
\begin{equation*}
T_{\partial g^*(\oy)\cap {\rm Im}\,K}(K\ox)=T_{\partial g^*(\oy)}(K\ox)\cap T_{{\rm Im}\, K}(K\ox)=T_{\partial g^*(\oy)}(K\ox)\cap\Im K.
\end{equation*}
This implies that 
\begin{equation*}
\begin{aligned}
K^{-1}\left(T_{\partial g^*(\oy)\cap {\rm Im}\,K}(K\ox)\right)=K^{-1}\left( T_{\partial g^*(\oy)}(K\ox)\cap \Im K\right)=K^{-1}\left(T_{\partial g^*(\oy)}(K\ox)\right), 
\end{aligned}
\end{equation*}
which justifies the equivalence between \eqref{eq:Nes} and \eqref{eq:Suf}. The proof is complete. 
\end{proof}

 Obviously,  condition {\bf (ii)} in Theorem~\ref{thm:NSC} is valid when $K$ is a surjective operator, i.e., $\Im K=\YY$ or $\Ker K^*=0$. A typical situation in least-squares regularized problems is when $K$ is the identity mapping and $\YY=\XX$. In this case, condition \eqref{eq:Suf} is known as a characterization for strong solution $\ox$ of problem $P(\bar b,\bar\mu)$ as in \cite[Theorem~3.3]{FNP23}.
\begin{Corollary}[Equivalence between isolated calmness and strong solution]\label{cor:Equ1} Suppose that $K:\XX\to \XX$ is the identity mapping. Let $\ox$ be an optimal solution of problem $P(\bar b,\bar \mu)$. If $g$ satisfies the quadratic growth condition at $\ox$ for $\oy=-\frac{1}{\bar \mu}\Phi^*(\Phi\ox-\bar b)$, the following are equivalent 
\begin{itemize}
    \item[{\bf(i)}] The solution mapping $S_{LS}$ has the isolated calmness property at $(\bar b,\bar \mu)$ for $\ox$.
    \item[{\bf(ii)}] $\ox$ is the strong solution of problem $P(\bar b,\bar \mu)$.
    \item[{\bf(iii)}] $\Ker \Phi\cap T_{\partial g^*(\oy)}(\ox)=\{0\}.$
\end{itemize}
\end{Corollary}
 \begin{proof} The equivalence between {\bf (i)} and {\bf(ii)} follows directly from Theorem~\ref{thm:NSC}. Moreover, the equivalence between {\bf (ii)} and {\bf(iii)} is obtained recently in \cite[Theorem~3.3]{FNP23} with the same assumption that $g$ satisfies the quadratic growth condition at $\ox$ for $\oy$.
 \end{proof}

  Another example of surjective linear operator is the one-dimensional {\em discrete gradient operator} $D:\R^n\to\R^{n-1}$: 
  \[
  Dx:=(x_1-x_2, x_2-x_3, \ldots, x_{n-1}-x_n)^T\quad \mbox{ for }\quad x\in \R^n,
  \]
  that is broadly used in signal processing.  However,  the well-known two-dimensional (2D) {\em discrete gradient operator} $\nabla:\R^{n_1\times n_2}\to \R^{2n_1n_2}$ defined below is definitely not surjective
\begin{equation*}\label{eq:TV}
    (\nabla x)_{i,j}:=\begin{pmatrix}
        (\nabla x)^1_{i,j}\\(\nabla x)^2_{i,j}
    \end{pmatrix}\quad \mbox{for}\quad x\in \R^{n_1\times n_2}\quad \mbox{with}
\end{equation*}
\begin{equation*}\label{eq:total}
   (\nabla x)^1_{i,j}:=\left\{\begin{array}{ll}x_{i+1,j}-x_{i,j}\quad &\mbox{for}\quad i < n_1\\ 
   0 \quad &\mbox{for}\quad i = n_1,
\end{array}\right.\; \mbox{and}\;\;\;\\
(\nabla x)^2_{i,j}:=\left\{\begin{array}{ll}x_{i,j+1}-x_{i,j}\quad &\mbox{for}\quad j < n_2\\ 
0 \quad &\mbox{for}\quad j= n_2.
\end{array}\right.
\end{equation*}
This linear operator is widely used in image processing in the so-called {\em isotropic total variation seminorm} regularizer 
\[
g(\nabla x)=\sum_{1\le i< n_1, 1\le j< n_2}\sqrt{(x_{i+1,j}-x_{i,j})^2+(x_{i,j+1}-x_{i,j})^2},
\]
where the outer function $g$ is the group Lasso regularizer \eqref{eq:GLasso} with $(n_1-1)(n_2-1)$ groups of two elements. Actually, any group Lasso regularizer satisfies the quadratic growth condition and $\partial g^*(\oy)$ is a polyhedral set for any $\oy\in \YY$, as observed from Remark~3.3(b). This allows us to obtain a similar result with Corollary~\ref{cor:Equ1} for the case of analysis group Lasso below.

\begin{Corollary}[Equivalence between isolated calmness and solution uniqueness for analysis group Lasso problems]\label{cor:Equ2} Suppose that $K:\XX\to \R^m$ is a linear operator and $g:\R^m\to\R$ is the group Lasso regularizer defined \eqref{eq:GLasso}. Let $\ox$ be an optimal solution of problem $P(\bar b,\bar \mu)$ and $\oy\in \partial g(K\ox)$ with $K^*\oy=-\frac{1}{\bar \mu}\Phi^*(\Phi \bar x-\bar b)$. The following are equivalent:  
\begin{itemize}
    \item[{\bf(i)}] The solution mapping $S_{LS}$ has the isolated calmness property at $(\bar b,\bar \mu)$ for $\ox$.
    \item[{\bf(ii)}] $\ox$ is the (strong) unique solution of problem $P(\bar b,\bar \mu)$.
    \item[{\bf(iii)}] $\Ker \Phi\cap K^{-1}\left(T_{\partial g^*(\oy)}(K\ox)\right)=\{0\}.$
\end{itemize}
\end{Corollary}
\begin{proof}
Let us start to prove that [{\bf(i)}$\Rightarrow${\bf(ii)}]. As $S_{LS}$ is isolated calm at $(\bar b, \bar \mu)$ for $\ox$, there exists a neighborhood $U$ of $\ox$ in $\XX$ such that $S_{LS}(\bar b,\bar\mu)\cap U=\{\ox\}$. As the least-squares regularized problem is a convex optimization problem, $\ox$ is the unique solution of $P(\bar b,\bar \mu)$. Since $\partial g^*(\oy)$ is polyhedral, the equivalence between {\bf (i)} and {\bf (iii)} follows from Theorem~\ref{thm:NSC}. 

It remains to justify [{\bf(ii)}$\Rightarrow${\bf(iii)}] by supposing that $\ox$ is the unique solution of $P(\bar b, \bar\mu)$. Pick any $w\in \Ker \Phi \cap K^{-1}\left(T_{\partial g^*(\oy)}(K\ox)\right)$, which implies $Kw\in T_{\partial g^*(\oy)}(K\ox) $. As $\partial g^*(\oy)$ is a polyhedral set, 
\[
T_{\partial g^*(\oy)}(K\ox)=\cone \left(\partial g^*(\oy) - K\ox \right).
\]
Hence, we find some $t>0$ such that $K(\ox +tw)\in \partial g^*(\oy).$ It follows that 
\[
\ox+tw\in K^{-1}(\partial g^* (\oy))=\partial h^*(\ov)\quad \mbox{with}\quad \ov=-\frac{1}{\bar\mu}\Phi^*(\Phi \ox-\bar b).
\]
Thus, $-\frac{1}{\bar\mu}\Phi^*(\Phi \ox-b)\in \partial h(\ox+tw)=K^*\partial g(K(\ox+tw))$. It results that   $\ox+tw$ is also optimal solution of $P(\bar b,\bar \mu)$, which deduces $w=0$ as $\ox$ is unique solution of $P(\bar b,\bar \mu)$. This verifies condition {\bf(iii)}. The proof is complete.
\end{proof}

The condition {\bf (i)} in Theorem~\ref{thm:NSC} is also valid for the class of  convex piecewise linear-quadratic  functions $g$, which also satisfy the primal-dual quadratic growth condition as  discussed in Remark~\ref{rem:QL}.  Indeed, \cite[Theorem~11.14 and Proposition~12.30]{RW98} tells us that $\partial g^*(\oy)$ is polyhedral for any $\oy\in \YY$ in this case. Note that $h(x)=g(Kx)$ is also a convex piecewise linear-quadratic  function. This observation together with Theorem~\ref{thm:NSC} gives us another result that is similar with Corollary~\ref{cor:Equ1} for this class of functions.

\begin{Corollary}[Equivalence between isolated calmness and solution uniqueness] Suppose that $g:\YY\to \R\cup\{\infty\}$ is a l.s.c. convex piecewise linear-quadratic function with $\dom g\cap\Im K\neq \emptyset$. The following are equivalent 
\begin{itemize}
    \item[{\bf (i)}] The solution mapping $S_{LS}$ has isolated calmness property at $(\bar b,\bar \mu)$ for $\ox$.
    \item[{\bf (ii)}] $\ox$ is  the (strong) unique solution of $S_{LS}(\bar b,\bar\mu)$.  
\end{itemize}
    
\end{Corollary}


\begin{Remark}[Analysis nuclear norm minimization problems]{\rm
Another important regularizer broadly used in optimization and compressed sensing is the nuclear norm $g(Y)=\|Y\|_*$, the sum of all singular values $\sigma_1(Y), \ldots, \sigma_m(Y)$ of $Y\in \YY=\R^{m\times n}$ ($m\le n$). Let $\ox$ is an optimal solution of $P(\bar b, \bar \mu)$, i.e., $\ov:=-\frac{1}{\bar\mu}\Phi^*(\Phi \ox-\bar b)\in K^*\partial g(K\ox)$ and $\OY\in \partial\|K\ox\|_*$ satisfy condition $-\frac{1}{\bar\mu}\Phi^*(\Phi \ox-\bar b)=K^*\OY$. It is well-known that $\OX:=K\ox$ and  $\OY$ have {\em simultaneous ordered singular value decomposition} \cite{LS05a} with orthogonal matrices $\OU\in \R^{m\times m}$ and $\OV\in\R^{n\times n}$ in the sense that 
\begin{equation*}
    \OX=\OU({\rm Diag}\,\sigma(\OX))\OV^T\qquad \mbox{and}\qquad \OY=\OU({\rm Diag}\,\sigma(\OY))\OV^T,
\end{equation*}
where $\sigma(\OX):=\left (\sigma_1(\OX), \ldots,\sigma_{m}(\OX)\right)^T\in \R^m_+$ contains all singular values of $\OX$ and 
\[
{\rm Diag}\,\sigma(\OX):=\begin{pmatrix}\sigma_1(\OX)&\ldots&0&0&\ldots& 0\\
0&\ddots&0&0&\ldots& 0\\ 
0&\ldots &\sigma_{m}(\OX)&0&\ldots 
&0\end{pmatrix}_{m\times n}.
\]
Moreover, the nuclear norm satisfies the primal-dual quadratic growth condition as discussed in Remark~\ref{rem:QL}(c).

Let  $r=\rank(\OX)$ be the rank of $\OX$, which is the number of positive singular values of $\OX$. The subdifferential of $g$ at $\OX$ is computed by
\[
\partial g(\OX)=\partial\|\OX\|_*=\left\{\OU\begin{pmatrix}\Id_{r\times r}&0\\0&W\end{pmatrix}_{m\times n}\OV^T|\; \|W\|\le 1\right\},
\]
where $\Id_{r\times r}$ is the $r\times r$ identity matrix and $\|W\|$ stands for the {\em spectral norm} of $W\in \R^{(m-r)\times(n-r)}$. Let $p$ be the number of singular values of $\OY$ that are equal to $1$. It follows that $r\le p\le \rank(\OY)$. \cite[Proposition~10]{ZS17} tells us that   
\[
\partial g^*(\OY)=\OU\begin{pmatrix}\mathbb{S}^p_+&0\\0&0\end{pmatrix}_{m\times n}\OV^T,
\]
where $\mathbb{S}^p_+$ is the set of all {\em positive semidefinite} $p\times p$ matrices. Obviously, $\partial g^*(\OY)$  is not a polyhedral, i.e., condition {\bf (i)} in Theorem~\ref{thm:NSC} fails. However, it is possible that condition {\bf (ii)} in  Theorem~\ref{thm:NSC} is satisfied. Indeed, note that 
\[
\ri \partial g^*(\OY)=\OU\begin{pmatrix}\mathbb{S}^p_{++}&0\\0&0\end{pmatrix}_{m\times n}\OV^T,
\]
where $\mathbb{S}^p_{++}$ is the set of all positive definite $p\times p$ matrices. In this case, condition {\bf(ii)} in  Theorem~\ref{thm:NSC} turns into
\begin{equation*}
    \Im K\cap \OU\begin{pmatrix}\mathbb{S}^p_{++}&0\\0&0\end{pmatrix}_{m\times n}\OV^T\neq \emptyset. 
\end{equation*}
When the so-called {\em nondegeneracy condition} holds, i.e., $p=r$, the above condition is automatic, as $\OX=K\ox$ trivially lies in the left intersection.   \hfill$\triangle$

}    
\end{Remark}

Next let us consider the following parametric optimization problem
\begin{equation}\label{p:PB}
    \min_{x\in \XX}\quad f(x,p)+g(Kx+v),
\end{equation}
which is problem~\eqref{p:PO} with an extra {\em canonical perturbation} $v\in \YY$ that allows some possible noise/ error in the computation of $Kx$. Define the optimal solution mapping of this problem with parameter $(p,v)$ by
\begin{equation}
    \Bar S (p,v):={\rm argmin}\, \{f(x,p)+g(Kx+v)|\, x\in \XX\}.
\end{equation}
It is obvious that $\Bar S(p,0)=S(p)$. 
With $(p,v)$ around $(\op,0)$ and any $x\in \Bar S(p,v)$ around $\ox$, the similar chain rule with \eqref{eq:CR} under the RCQ \eqref{eq:RCQ} provides us a multiplier $y\in \partial g(Kx+v)$ such that $0\in \nabla_x f(x,p)+K^*y$.
As $Kx+v\in \partial g^*(y)$, the latter is equivalent to the following system
\begin{equation}\label{eq:GE}
\begin{pmatrix}0\\v\end{pmatrix}\in \begin{pmatrix}\nabla_x f(x,p)+K^*y\\-Kx\end{pmatrix}+\begin{pmatrix}0\\\partial g^*(y)\end{pmatrix},
\end{equation}
when $(p,v,x)$ is around $(\op,0,\ox)$. This is usually referred as the primal-dual system of problem~\eqref{p:PB}. Let us define the solution mapping of this system as 
\begin{equation}\label{eq:TS}
\Tilde S(p,v):=\{(x,y)\in \XX\times \YY \mbox{ satisfying \eqref{eq:GE}}\}.
\end{equation}
For the rest of this section, we pay our attention to studying the isolated calmness of this primal-dual solution mapping. Note that the linearized system for \eqref{eq:GE} is 
\begin{equation}\label{eq:LGE}
\begin{pmatrix} u\\v
\end{pmatrix}\in \begin{pmatrix}\nabla_x f(\ox,\op)+\nabla^2_{xx} f(\ox,\op)(x-\ox)+K^*y\\-Kx\end{pmatrix}+\begin{pmatrix}0\\\partial g^*(y)\end{pmatrix}\quad \mbox{with}\quad (u,v)\in \XX\times \YY.
\end{equation}
Define 
\begin{equation}\label{eq:HS}
    \Hat S(u,v):=\{(x,y)\in \XX\times \YY \mbox{ satisfying the above system}\}
\end{equation} 
and 
\[
F(x,y):= (\nabla_xf(\ox,\op)+\nabla^2_{xx} f(\ox,\op)(x-\ox)+K^*y, -Kx+\partial g^*(y))\quad \mbox{for}\quad (x,y)\in \XX\times \YY.
\]
Note that $\Hat S^{-1}(x,y)=F(x,y)$. To study the isolated calmness of $\Hat S$ at $(\ou,\ov)=(0,0)$ for some $(\ox,\oy)\in \Hat S(\ou,\ov)$, we need the Levy-Rockafellar criterion \eqref{eq:LR0} again
\begin{equation}\label{eq:LR}
0\in DF ((\ox,\oy)|\,(\ou,\ov))(x,y)\Longrightarrow(x,y)=0.
\end{equation}
Note that 
\[
DF ((\ox,\oy)|\,(\ou,\ov))(x,y)=\left(\nabla^2_{xx} f(\ox,\op)x+K^*y,-Kx+D\partial g^*(\oy|\; K\ox)(y)\right).
\]
Hence $0\in DF((\ox,\oy)|\,(\ou,\ov))(x,y)$ if and only if 
\begin{equation}\label{eq:LR2}
\nabla^2_{xx} f(\ox,\op)x+K^*y=0\quad \mbox{and}\quad Kx\in D\partial g^*(\oy|\; K\ox)(y).
\end{equation}
It follows that 
\begin{equation}\label{eq:Szero}
0=\la \nabla^2_{xx} f(\ox,\op)x+K^*y,x\ra=\la \nabla^2_{xx} f(\ox,\op)x,x\ra+\la y,Kx\ra. 
\end{equation}
Since $f$ is a convex function on $x$, $\nabla^2_{xx} f(\ox,\op)$ is positive semidefinite. Moreover, since $Kx\in D\partial g^*(\oy|\; K\ox)(y)$ and $g^*$ is a convex function, we have $\la Kx,y\ra\ge0$ by \eqref{eq:SPD0}. Equation \eqref{eq:Szero} means 
\[
\la \nabla^2_{xx} f(\ox,\op)x,x\ra=0\quad \mbox{and}\quad \la y,Kx\ra=0,
\]
which implies that $x\in \Ker \nabla^2_{xx} f(\ox,\op)$. This together with the first equation in \eqref{eq:LR2} implies that $y\in \Ker K^*$.  As $Kx\in D\partial g^*(\oy|\; K\ox)(y)$ and $\la y,Kx\ra=0$, we are in the position of using the zero-product property of subgradient graphical derivative from Theorem~\ref{thm:Zero}. These analysis allow us to establish a characterization for the isolated calmness property for the solution mapping $\Hat S$ in \eqref{eq:HS} as follows.


\begin{Theorem}[Isolated calmness for the primal-dual system]\label{thm:calm} Let  $(\ox,\oy)\in \Bar S(\op,0)$. If  the solution mapping $\Hat S$ in \eqref{eq:HS} is isolated calm at $(\ou,\ov)$ for $(\ox,\oy)$, then we have
\begin{equation}\label{con:Calm}
  \Ker K^*\cap T_{\partial g(K\ox)}(\oy)=\{0\}\quad \mbox{and}\quad \Ker \nabla^2_{xx} f(\ox,\op)\cap K^{-1}\left(T_{\partial g^*(\oy)}(K\ox)\right)=\{0\}. 
\end{equation}
If, additionally, $g$  satisfies the primal-dual quadratic growth condition at $K\ox$ for $\oy$, then the above two conditions are sufficient for the isolated calmness of $\Hat S$ at $(\ou,\ov)$ for $(\ox,\oy)$. Consequently, they are sufficient for the isolated calmness of solution mapping $\Tilde S$ in \eqref{eq:TS} at $(\op,0)$ for $(\ox,\oy)$. 
\end{Theorem}
\begin{proof} 
Let us start to show that \eqref{con:Calm} is necessary for the isolated calmness of $\Hat S$ at $(\ou,\ov)$ for $(\ox,\oy)$. Suppose that $\Hat S$ is isolated calm at $(\ou,\ov)$ for $(\ox,\oy)$. Hence condition~\eqref{eq:LR2} is satisfied. Pick any $y\in \Ker K^*\cap T_{\partial g(K\ox)}(\oy)$ and observe from  inclusion \eqref{eq:KerDS} (with $F=\partial g^*$) that 
\[
K^*y=0\quad \mbox{and}\quad 0\in D\partial g^*(\oy| \,K\ox)(y).
\]
Thus, the pair $(0,y)\in \XX\times \YY$ satisfies \eqref{eq:LR2}, which makes $y=0$, i.e., $\{0\}=\Ker K^*\cap T_{\partial g(K\ox)}(\oy)$. Now let us pick any $x\in \Ker \nabla^2_{xx} f(\ox,\op)\cap K^{-1}\left(T_{\partial g^*(\oy)}(K\ox)\right)$. Derive from $\eqref{eq:KerDS}$ (with $F=\partial g$) that 
\[
\nabla^2_{xx} f(\ox,\op)x=0\quad \mbox{and}\quad Kx\in D\partial g^*(\oy| \,K\ox)(0),
\]
which tells us that the pair $(x,0)\in \XX\times \YY$ satisfies system \eqref{eq:LR2}. Hence $x=0$ by \eqref{eq:LR}. The two conditions in \eqref{con:Calm} are verified.

Next, we show that both conditions in \eqref{con:Calm} are sufficient for the isolated calmness of $\Hat S$ at $(\ou,\ov)$ for $(\ox,\oy)$, provided that $g$ satisfies the primal-dual quadratic growth condition at $K\ox$ for $\oy$. Indeed, suppose that $(x,y)$ satisfies \eqref{eq:LR2} or \eqref{eq:LR}. The discussion before this theorem tells us that $x\in \Ker \nabla^2_{xx}f(\ox,\op)$, $y\in \Ker K^*$, and that 
\[
y\in D\partial g(K\ox|\; \oy)(Kx)\quad\mbox{and}\quad \la y,Kx\ra=0.
\]
As $g$ satisfies the primal-dual quadratic growth condition at $K\ox$ for $\oy$, the zero-product property in Theorem~\ref{thm:Zero} leads us to 
\[
y\in T_{\partial g(K\ox)}(\oy)\quad  \mbox{and}\quad Kx\in T_{\partial g^*(\oy)}(K\ox).
\]
It follows that 
\[
y\in \Ker K^*\cap T_{\partial g(K\ox)}(\oy)\quad \mbox{and}\quad x\in \Ker \nabla^2_{xx} f(\ox,\op)\cap K^{-1}\left(T_{\partial g^*(\oy)}(K\ox)\right).
\]
According to \eqref{con:Calm}, we have $(x,y)=(0,0)$. Thus, the Levy-Rockafellar criterion \eqref{eq:LR} holds, i.e., $\Hat S$ has the isolated calmness property at $(\ou,\ov)$ for $(\ox,\oy)$. 

Finally, as the isolated calmness of $\Hat S$ at $(\ou,\ov)$ for $(\ox,\oy)$ from the linearized system \eqref{eq:LGE} implies the isolated calmness of $\Tilde  S$ at $(\op,0)$ for $(\ox,\oy)$ \cite[Theorem~3I.13]{DR09}, conditions in \eqref{con:Calm} are also sufficient for the latter property of $\Tilde S$.
\end{proof}

\begin{Remark}[Strict Robinson Constraint Qualification]{\rm 
The dual (equivalent) version of the first condition in \eqref{con:Calm} is 
\begin{equation}\label{def:SRCQ}
    \Im K+N_{\partial g(K\ox)}(\oy)=\YY.
\end{equation}
As the function $g$ is convex and $\oy\in \partial g(K\ox)\neq \emptyset$, we get from \cite[Proposition~2.126]{BS00}
\begin{equation}\label{eq:Max}
\sup\left\{\la v,y\ra|\, y\in  \partial g(K\ox)\right\}=dg(K\ox|\, v)=\lim_{t\dn 0}\dfrac{g(K\ox+tv)-g(K\ox)}{t}, 
\end{equation}
which is the {\em directional derivative} of $g$ at $K\ox$ for the direction $v\in \YY$. Thus, the normal cone $N_{\partial g(K\ox)}(\oy)$ is actually the {\em critical cone} of $g$ at $K\ox$ for $\oy$ in the sense that 
\begin{equation}\label{eq:CC1}
N_{\partial g(K\ox)}(\oy)=\{v\in \YY|\, dg(K\ox|\, v)=\la v, \oy\ra\}\subset \dom dg(K\ox|\, \cdot)\subset T_{{\rm dom}\, g}(K\ox).
\end{equation}
Condition~\eqref{def:SRCQ} can be refereed as the {\em Strict Robinson Constraint Qualification} (SRCQ) for the composite function $g(Kx)$ at $K\ox$ for the multiplier $\oy\in \partial g(K\ox)$; see, e.g., \cite{DR97,ZZ16,DSZ17,CS18} and \cite[Definition~4.46]{BS00} for some special choices of $g$ such as the indicator of of a closed convex set $C$ of $\YY$. Due to the above inclusions, the SRCQ in \eqref{def:SRCQ} is stronger than the RCQ in \eqref{eq:RCQ}. The appearance of SRCQ for our composite model is consistent with other works \cite{DR97,ZZ16,DSZ17,CS18} mostly for some particular constrained optimization  problems. Here, our function $g$ is in a general setting and  just needs to satisfy the primal-dual quadratic growth condition. 

The dual (equivalent) version of the second condition in \eqref{con:Calm} or \eqref{eq:Suf} is 
\begin{equation}\label{def:DSRCQ}
    \Im \nabla^2_{xx}f(\ox,\op)^*+K^*\left(N_{\partial g^*(\oy)}(K\ox)\right)=\XX.
\end{equation}
Similarly to \eqref{eq:CC1}, the normal cone $N_{\partial g^*(\oy)}(K\ox)$ is also the critical cone of the conjugate function $g^*$ at $\oy$ for $K\ox$, as
\begin{equation}\label{eq:CC2}
N_{\partial g^*(\oy)}(K\ox)=\{u\in \YY|\, dg^*(\oy|\, u)=\la K\ox,u\ra\}. 
\end{equation}
This condition is also refereed as the Strict Robinson Constraint Qualification for the {\em dual} problem of the following model in \cite{CS18}
\begin{equation}\label{p:DS}
  Q(U,w)\quad   \min_{X\in \R^{m\times n}}\quad k(\mathcal{B}X)+ \theta(X)-\la U,X\ra \quad \mbox{subject to}\quad \mathcal{C}X-d+w\in \Theta,
\end{equation}
where $\mathcal{B}: \R^{m\times n}\to \mathbb{Z}$ and $\mathcal{C}: \R^{m\times n}\to \mathbb{W}$ are linear operators between Euclidean spaces, $d\in \mathbb{W}$ is a known vector, $k:\mathbb{Z}\to \R$ is a {\em strictly} convex and twice differentiable function, $\theta(X)=\|X\|_*$ is the nuclear norm for $X\in \R^{m\times n}$, $\Theta \subseteq \mathbb{W}$ is a convex {\em polyhedral cone}, and $(U,w)\in \R^{m\times n}\times \mathbb{W}$ represents for the full perturbation (parameter). This problem is a particular case of our problem~\eqref{p:PB} by with 
$\XX:=\R^{m\times n}$, $\mathbb{P}:=\XX$, 
\begin{equation*}
     f(X,U):=k(\mathcal{B}X)-\la U,X\ra,\; \mathcal{K}X:=(X, \mathcal{C}X)\in \YY:=\R^{m\times n}\times \mathbb{W},\;\;g(X,w):=\theta(X)+\delta_{\Theta}(w-d),
\end{equation*}
and $v:=(0,w)\in \YY$ is the canonical parameter in \eqref{p:PB}. Suppose $\OX\in \XX$ is an optimal solution of problem $Q(0,0)$ in \eqref{p:DS} and $\oy=(\OY,\ow)\in \YY$ is a multiplier satisfying $\OY\in \theta(\OX)$  and $\mathcal{C}^*\ow\in N_\Theta(\mathcal{C}\OX-d)$.  As the function $k$ is strictly convex, $\Ker \nabla^2_{xx}f(\OX,0)=\Ker \mathcal{B}$. Moreover, since the function $g$ is separable with variable $X$ and $w$, condition \eqref{def:DSRCQ} turns  into 
\begin{equation}\label{def:SRCQ2}
\Im \mathcal{B}^*+N_{\partial \theta^*(\OY)}(\OX)+\mathcal{C}^*\left(N_{\partial \delta_\Theta^*(\ow)}(\mathcal{C}\OX-d)\right)=\XX,
\end{equation}
which is exactly the SRCQ condition for the dual problem defined in \cite[Condition~(5.8)]{CS18} after imposing formula \eqref{eq:CC2}. Consequently, the equivalence between the isolated calmness of $\Tilde S$ at $(0,0)\in \YY$ for $(\ox,\oy)$ and SRCQ conditions~\eqref{def:SRCQ} and \eqref{def:SRCQ2} can be obtained from our Theorem~\ref{thm:calm} (the tilt perturbation $-\la U,X\ra$ instead of general parameter plays an important role here and the linearization as in \eqref{eq:LGE} is not necessary any more.) This covers some important part of \cite[Theorem~5.3]{CS18}. To avoid dilution, we do not go into details about this implication, but the key observation is the (separable) function $g(X,w)=\theta(X)+\delta_{\Theta}(w-d)$ defined above satisfies the primal-dual quadratic growth condition at $(\OX,0)$ for $(\OY,\ow)$, as discussed in Remark~\ref{rem:QL} (a)\&(c). This observation also allows us to extend model \eqref{p:DS} to more general situations, at which the functions $\theta$ and $\delta_\Theta$ just need to satisfy the primal-dual quadratic growth conditions; particularly, $\theta$ could be any  spectral functions from Remark~\ref{rem:QL}(c) and $\Theta$ is the positive semi-definite cone $\mathbb{S}_+^r$ (the indicator function $\delta_{\mathbb{S}_+^r}$ is also a spectral function from Remark~\ref{rem:QL}(c).)\endproof
}
\end{Remark}

For the least-squares regularized problems \eqref{eq:LeS}, its corresponding primal-dual solution mapping is 
\begin{equation}\label{eq:PDLS}
   \Tilde S_{LS}(b,\mu,v):=\left\{(x,y)\in \XX\times \YY|\; \begin{pmatrix}0\\v\end{pmatrix}\in\begin{pmatrix} 0&K^*\\-K&0\end{pmatrix}\begin{pmatrix}x\\y\end{pmatrix}+ \begin{pmatrix}\frac{1}{\mu}\Phi^*(\Phi x-b)\\\partial g^*(y)\end{pmatrix}\right\}. 
\end{equation}
Suppose that $(\ox,\oy)\in \Tilde S_{LS}(\bar b,\bar \mu,0)$. 
Conditions in \eqref{con:Calm} together with \eqref{eq:KHes} turn into 
\begin{equation}\label{eq:CalmLS}
    \Ker K^*\cap T_{\partial g(K\ox)}(\oy)=\{0\}\quad \mbox{and}\quad \Ker \Phi\cap K^{-1}\left(T_{\partial g^*(\bar y)}(K\ox)\right)=\{0\}, 
\end{equation}
which are sufficient for the isolated calmness of $\Tilde S_{LS}$ at $(\bar b,\bar \mu,0)$ for $(\ox,\oy)$. Whether these two conditions are also necessary for the isolated calmness of $\Tilde S_{LS}$ at $(\bar b,\bar \mu,0)$ is a natural question in the spirit of Theorem~\ref{thm:NSC}. As the first condition in \eqref{eq:CalmLS} only involves information of $g$ and $K$, it plays a role of a qualification condition. On the other hand, the second condition in \eqref{eq:CalmLS} is exactly \eqref{eq:Suf}, which is a full characterization for the isolated calmness of $S_{LS}$ at $(\bar b,\bar \mu)$ under either {\bf (i)} or {\bf (ii)} in Theorem~\ref{thm:NSC}. We conclude this section with the following result, which provides the equivalence between the isolated calmness of $\Tilde S_{LS}$ at $(\bar b,\bar \mu,0)$ for $(\ox,\oy)\in \Tilde S(\bar b,\bar \mu,0)$ under several qualification conditions. 

\begin{Corollary}[Isolated calmness of the primal-dual solution mapping $\Tilde S_{LS}$]\label{cor:Eqi3} Let $(\ox,\oy)\in \Tilde S_{LS}(\bar b,\bar \mu,0)$ be a solution of the primal-dual system in \eqref{eq:PDLS}. Suppose that $g$ satisfies the primal-dual quadratic growth condition at $K\ox$ for $\oy$. Then $\Tilde S_{LS}$ is isolated calm at $(\bar b,\bar \mu,0)$ for $(\ox,\oy)$ if and only if condition~\eqref{eq:Suf} is satisfied, provided that one of the following conditions holds:

\begin{itemize}
    \item[{\bf(i)}] The set $\partial g^*(\oy)$ is polyhedral and $\Ker K^*\cap T_{\partial g(K\ox)}(\oy)=\{0\}$. 
     \item[{\bf(ii)}] $\Im K\cap\ri(\partial g^*(\oy))\neq \emptyset$ and $\Ker K^*\cap T_{\partial g(K\ox)}(\oy)=\{0\}$. 
     \item[{\bf(iii)}] $\Ker K^*\cap N_{\partial g^*(\oy)}(K\ox)=\{0\}$. 
\end{itemize}
    
\end{Corollary}
\begin{proof} Note that if $(x,y)\in \Tilde S_{LS}(b,\mu,0)$, then $x\in S_{LS}(b,\mu)$. Thus if $\Tilde S_{LS}$ is isolated calm at $(\bar b,\bar \mu,0)$ for $(\ox,\oy)$, then the solution mapping $S_{LS}$ is isolated calm at $(\bar b,\bar \mu)$ for $\ox$. Combining Theorem~\ref{thm:NSC} and Theorem~\ref{thm:calm} gives us the equivalence between the isolated calmness of $\Tilde S_{LS}$ at $(\bar b,\bar \mu,0)$ for $(\ox,\oy)$ and condition~\eqref{eq:Suf} under the validity of either {\bf(i)} or {\bf (ii)}. 

To justify this equivalence under {\bf (iii)}, we only need to show that {\bf (iii)} is sufficient for {\bf(ii)}. We claim that $T_{\partial g(K\ox)}(\oy)\subset N_{\partial g^*(\oy)}(K\ox)$, or equivalently, $T_{\partial g^*(\oy)}(K\ox)\subset N_{\partial g(K\ox)}(\oy)$ by taking their polar cones. Indeed, for any $z\in T_{\partial g^*(\oy)}(K\ox)$, there exist $t_k\dn 0$ and $z_k\to z$ such that $K\ox+t_kz_k\in \partial g^*(\oy)$. It follows that 
\[
g(K\ox+t_kz_k)=\la K\ox+t_kz_k,\oy\ra-g^*(\oy)=t_k\la z_k,\oy\ra+\la K\ox,\oy\ra-g^*(\oy)=t_k\la z_k,\oy\ra+g(K\ox).
\]
Since $g$ is a convex function, we know that 
\[
\la z,\oy\ra=\lim_{k\to \infty} \la z_k,\oy\ra=\lim_{k\to\infty}\dfrac{g(K\ox+t_kz_k)-g(K\ox)}{t_k}=dg(K\ox;z),
\]
As observed in \eqref{eq:CC1}, we have  $z\in N_{\partial g(K\ox)}(\oy)$. This verifies the claim.

Now, suppose that {\bf (iii)} is satisfied. The above claim leads us to $\Ker K^*\cap T_{\partial g(K\ox)}(\oy)=\{0\}$. Next, we show that $\Im K\cap \ri(\partial g^*(\oy))\neq \emptyset$. Indeed, condition {\bf (iii)} is  equivalent to the following form \cite[Corollary~2.98]{BS00}
\[
0\in {\rm int}\, (K\ox+\Im K-\partial g^*(\oy))={\rm int}\, (\Im K-\partial g^*(\oy)),
\]
which implies that 
\[
0\in \ri(\Im K-\partial g^*(\oy))=\ri(\Im K)-\ri(\partial g^*(\oy))=\Im K-\ri(\partial g^*(\oy)).
\]
This means $\Im K\cap \ri(\partial g^*(\oy))\neq \emptyset$. Hence, condition {\bf (ii)} is valid whenever condition {\bf (iii)} holds. The proof is completed.  
\end{proof}

\section{Conclusion}
In this paper, we give a systematic study for the isolated calmness of  solution mappings of the composite model \eqref{p:PO} 
and the Lagrange system of \eqref{p:PB} with canonical parameter.
Our main intension is to apply these models to the case of least-squares optimization problems in  \eqref{eq:LeS} and \eqref{eq:PDLS}. We obtain several necessary and sufficient conditions for isolated calmness of these solution mappings. The gaps between those necessary and sufficient conditions are narrow. Indeed, they are the same under mild conditions.   

It may be interesting to extend this approach to Lipschitz stability \cite{BBH23,BBH24,CHNS25,N24} for these solution mappings, especially for the convex composite models that the aforementioned papers have not discovered yet. Another direction that we are working on is for the case that $\bar \mu =0$ in \eqref{eq:LeS}, which reduces the problem to the {\em regularized linear inverse problem}. The isolated calmness is very close 
to the so-called property of {\em stable recovery} widely known in Inverse Problems community; see, e.g., \cite{NPV24} for a recent development in this direction.


\begin{thebibliography}{99} 

\bibitem{AG08} F. J. Artacho Arag\'on and M. H. Geoffroy: Characterization of metric regularity of subdifferentials, {\em J. Convex Anal.}, {\bf 15}(2), 365--380, 2008. 


\bibitem{BM24} M. Benko and P. Mehlitz: Isolated calmness of perturbation mappings and superlinear convergence of Newton-Type methods, {\em J. Optim. Theory Appl.}, {\bf 203}, 1587--1621, 2024.

\bibitem{BBH23} A. Berk, S.  Brugiapaglia, and T. Hoheisel: Lasso reloaded: a variational analysis perspective with applications to compressed sensing, {\em SIAM J. Math. Data Sci.}, {\bf 5}, 1102--1129,  2023.

 \bibitem{BBH24}   A. Berk, S.  Brugiapaglia, and T. Hoheisel: Square Root LASSO: well-posedness, Lipschitz stability and the tuning trade off, {\em SIAM J. Optim.}, {\bf 34}(3), 2609--2637, 2024.


\bibitem{BLPS21} J. Bolte, T. Le, E. Pauwels, and A. Silveti-Falls: Nonsmooth implicit differentiation for machine learning and optimization, {\em Advances in Neural Information Processing Systems} (NeurIPS 2021). 34:13537--13549.

\bibitem{BPS22} J. Bolte, E. Pauwels, and A. Silveti-Falls: Differentiating nonsmooth solutions to parametric monotone inclusion problems,  {\em SIAM J. Optim.}, {\bf 34}, 71--97,  2024.
 

\bibitem{BS00}
J. F. Bonnans and A. Shapiro:
 {\em Perturbation Analysis of Optimization Problems}, Springer Series in Operations Research. (Springer-Verlag, New York), 2000.
\bibitem{BM09}
J. M. Borwein and W. B. Moors: Stability of closedness of convex cones under linear mappings, {\em J. Convex Anal.}, {\bf 16}, 699--705, 2009.
 
 
\bibitem{BR65} A. Br{\o}ndsted and R.T.  Rockafellar: On the subdifferentiability of convex functions,    {\em Proc. AMS}, {\bf 16}(4),  605--611, 1965.






\bibitem{CHNT21}
N.~H. Chieu, L.~V. Hien, T.~T.~A. Nghia, and H. A. Tuan.
\newblock 	
Quadratic growth and strong metric subregularity of the subdifferential via subgradient graphical derivative,
\newblock {\em  SIAM J. Optim.}, {\bf 31}(1), 545--568, 2021.

\bibitem{CDZ17}  Y.  Cui,  C. Ding,  X. Zhao: Quadratic growth conditions for convex matrix optimization problems associated with spectral functions, SIAM J. Optim.,  {\bf 27}(4), 2332--2355, 2017.


\bibitem{CHNS25} Y. Cui, T. Hoheisel, T. T. A. Nghia, and D. Sun: Lipschitz stability of  least-squares problems regularized by functions with $\mathcal{C}^2$-cone reducible conjugates, {\em Math. Oper. Res.}, 2026 (to appear). arXiv:2409.13118


\bibitem{CS18} Y. Cui and D. Sun: A complete characterization of the robust isolated calmness of nuclear norm regularized convex optimization problems. {\em J. Comput. Math.}, {\bf 36}(3):441--458, 2018.


\bibitem{DSZ17} C. Ding, D. Sun, and L. Zhang: Characterization of the robust isolated calmness for a class of conic programming problems, {\em SIAM J. Optim.}, {\bf 27}(1), 67--90, 2017.


\bibitem{DR09} A. L. Dontchev and R. T. Rockafellar: \textit{Implicit Functions and Solution Mappings}, \textit{A View from Variational Analysis}, Springer, Dordrecht, 2009.

\bibitem{DR97} A. L. Dontchev and R. T. Rockafellar: Characterizations of Lipschitzian stability in nonlinear programming,  {\em Mathematical programming with data perturbations. CRC Press}, A.V. Fiacco, ed.), Marcel Dekker, New York, 65--82, 1997.

\bibitem{FNP23}  J. Fadili, T. T. A. Nghia, and D. N. Phan:  Geometric characterizations of strong minima with applications to nuclear norm minimization problems,    arXiv:2308.09224, 2023. 

\bibitem{GO16} H. Gfrerer and J. V. Outrata: On Lipschitzian properties of implicit multifunctions, {\em  SIAM J. Optim.}, {\bf 26}, 2160--2189, 2016.

\bibitem{HYZ24} C. Hu, W. Yao W, J. Zhang:
Decomposition method for Lipschitz stability of general LASSO-type problems, 2024. arXiv:2407.18884

\bibitem{I17} A. D. Ioffe: {\em Variational analysis of regular mappings}, Springer Monographs in Mathematics. Springer, Cham, 2017

\bibitem{KK02} D. Klatte and B. Kummer:  {\em Nonsmooth Equations in Optimization: Regularity, Calculus, Methods and Applications},
Kluwer, 2002.

\bibitem{LPB21} Y. Liu, S. Pan, and S. Bi: Isolated calmness of solution mappings and exact recovery conditions for nuclear norm optimization problems, {\em Optimization}, {\bf 70}, 481--510, 2021. 

\bibitem{LS05a} A. S. Lewis and H. S. Sendov: Nonsmooth analysis of singular values. I. Theory.
\textit{Set-Valued Anal.}, {\bf 13}, 213--241, 2005.


\bibitem{L96} A. B. Levy: Implicit multifunction theorems for the sensitivity analysis of variational conditions. Mathematical programming, {\bf 74}(3), 333--350, 1996.



\bibitem{MWY24}  K. Meng, P. Wu, and X. Yang: Lipschitz continuity of solution multifunctions of extended $\ell_1$ regularization problems, 2024. arXiv:2406.16053. 

\bibitem{M1} B. S. Mordukhovich: {\em Variational Analysis and Generalized Differentiation}, Springer, 2006.

\bibitem{MS20} A. Mohammadian and  M. E.  Sarabi: Twice epi-differentiability of extended-real-valued functions with applications in composite optimization. {\em SIAM J.  Optim.}, 30(3), 2379-2409, 2020.

\bibitem{MMS22} A. Mohammadi, B. S. Mordukhovich, and  M. E. Sarabi: Variational analysis of composite models with applications to continuous optimization, {\em  Math. Oper. Res.}, {\bf 47}(1), 397--426, 2022.

\bibitem{N24}  T. T. A. Nghia: Geometric characterizations of Lipschitz stability for convex optimization problems, {\em SIAM J. Optim.}   {\bf 35}(2), 927--958, 2025.

\bibitem{NPV24} T. T. A. Nghia, H. N. Pham, and N. V. Vo: Stable recovery of regularized linear inverse problems, {\em Inverse Problems}, {\bf 41}, 065018, 2025

\bibitem{R23b} 
 Rockafellar:
\newblock Convergence of augmented Lagrangian methods in extensions beyond nonlinear programming. 
\newblock{\em Math. Program.}, {\bf 199}(1): 375--420, 2023.

\bibitem{R70} R.~T. Rockafellar: {\em Convex Analysis}, Princeton University Press, Princeton, New Jersey, 1970.

\bibitem{RW98} R. T. Rockafellar and  R. J-B. Wets: {\em Variational Analysis}, Springer, Berlin, 1998.


\bibitem{VDFPD17}  S. Vaiter,  C. Deledalle, J.  Fadili, G. Peyr\'e, C. Dossal: The degrees of freedom of partly smooth regularizers, {\em Ann. Inst. Statist. Math.},  {\bf 69}, 791--832, 2017.

\bibitem{ZS17} Z. Zhou and A. M-C.  So: A unified approach to error bounds for structured convex optimization, \textit{Math. Program.}, {\bf 165}, 
689--728, 2017.

\bibitem{ZT95}
R. Zhang and J. Treiman:
 Upper-Lipschitz multifunctions and inverse subdifferentials, \textit{\em Nonlinear Anal.}, {\bf 24}, 273--286, 1995.

\bibitem{ZZ16} Y. L. Zhang and L. W. Zhang: On the upper Lipschitz property of the KKT mapping for
nonlinear semidefinite optimization, {\em Operations Research Letters}, {\bf 44},  474--478, 2016.

\end{thebibliography}
\end{document}